\documentclass[11pt,draftcls]{IEEEtran}{\onecolumn}
\usepackage{amsthm}
\usepackage{amsmath}
\usepackage{slashbox}
\usepackage{graphicx}
\usepackage{epstopdf}
\usepackage{epsfig,graphics,subfigure,psfrag,amsmath,amssymb}
\usepackage[nooneline,flushleft]{caption2}
\usepackage{cite}
\usepackage{float}
\usepackage{tocvsec2}
\usepackage{color}
\usepackage{verbatim}
\usepackage{algorithm, algorithmic}
\usepackage{mathtools}
\usepackage{hhline}
\usepackage{enumerate}
\usepackage{arydshln}

\newtheorem{thm}{Theorem}
\newtheorem{lem}{Lemma}
\newtheorem{prop}{Proposition}

\newtheorem{rem}{Remark}

\newtheorem{assump}{Assumption}

\makeatletter

\newcommand{\Rmnum}[1]{\expandafter\@slowromancap\romannumeral #1@}
\makeatother

\begin{document}
\title{Distributed Linearized ADMM for Network Cost Minimization}
\author{Xuanyu Cao and K. J. Ray Liu, \emph{Fellow, IEEE}\\
Email: \{apogne, kjrliu\}@umd.edu\\
Department of Electrical and Computer Engineering, University of Maryland, College Park, MD\\}
\maketitle
\begin{abstract}
In this work, we study a generic network cost minimization problem, in which every node has a local decision vector to determine. Each node incurs a cost depending on its decision vector and each link also incurs a cost depending on the decision vectors of its two end nodes. All nodes cooperate to minimize the overall network cost. The formulated network cost minimization problem has broad applications in distributed signal processing and control over multi-agent systems. To obtain a decentralized algorithm for the formulated problem, we resort to the distributed alternating direction method of multipliers (DADMM). However, each iteration of the DADMM involves solving a local optimization problem at each node, leading to intractable computational burden in many circumstances. As such, we propose a distributed linearized ADMM (DLADMM) algorithm for network cost minimization. In the DLADMM, each iteration only involves closed-form computations and avoids local optimization problems, which greatly reduces the computational complexity compared to the DADMM. We prove that the DLADMM converges to an optimal point when the local cost functions are convex and have Lipschitz continuous gradients. Linear convergence rate of the DLADMM is also established if the local cost functions are further strongly convex. Numerical experiments are conducted to corroborate the effectiveness of the DLADMM and we observe that the DLADMM has similar convergence performance as DADMM does while the former enjoys much lower computational overhead. The impact of network topology, connectivity and algorithm parameters are also investigated through simulations.
\end{abstract}


\section{Introduction}

The last decade has witnessed the advances of decentralized signal processing and control over networked multi-agent systems, which result in great research interest in distributed optimizations over networks. Such distributed optimization problems arise in fields such as adaptive signal processing over networks \cite{sayed2014adaptive}, distributed estimation over sensor networks \cite{dimakis2010gossip,predd2009collaborative}, decentralized power system state estimation and management \cite{guo2017hierarchical,gan2013optimal} as well as signal processing for communication networks \cite{shen2012distributed,huang2006distributed}. In these applications, data are distributed over individual nodes across the network. Centralized data processing and optimization suffer from high or even prohibitive communication overload and are vulnerable to link failures and network congestions. As such, optimizing and processing data in a decentralized manner, where only local information exchange among neighbors is allowed, are more favorable.

In the literature, distributed optimization has been extensively studied recently. Two important categories of distributed optimization problems are distributed network utility maximization (NUM) and consensus optimization. In distributed NUM, each agent has a local decision variable, based on which it obtains some utility. Agents cooperatively maximize the total utilities of the network subject to some coupling resource constraints such as the link capacity constraint in communication networks. For NUM, Wei \emph{et al.} propose and analyze a distributed Newton method in \cite{wei2013distributed1,wei2013distributed2}, while the effect of noisy information exchange is studied in \cite{zhang2008impact}. Moreover, Niu and Li present an asynchronous decentralized algorithm with elegant pricing interpretations for NUM \cite{niu2016asynchronous}. On the other hand, in consensus optimization, all agents share the same decision variable but have different local cost functions and the goal is to cooperatively minimize the total cost of the network. Nedic and Ozdaglar propose a decentralized subgradient method for consensus optimization in \cite{nedic2009distributed} while a dual averaging method is presented in \cite{duchi2012dual}. Specific forms of consensus problems such as adaptive signal processing over networks \cite{sayed2014adaptive} and average consensus (where agents cooperate to compute the average of individuals' data) \cite{erseghe2011fast} have been studied by using the alternating direction method of multipliers (ADMM). More recently, the general form of consensus problem is investigated by using the distributed ADMM (DADMM) in \cite{shi2014linear}, where linear convergence rate is established under some technical conditions. Later, several variants of DADMM are proposed for the consensus problems, including linearized ADMM \cite{ling2015dlm}, quadratically approximated ADMM \cite{mokhtari2016dqm} and dynamic ADMM \cite{ling2014decentralized}.

In all the aforementioned works, only costs or utilities at individual nodes are taken into consideration while the costs or gains of links are ignored. For example, for consensus optimization, the network cost is only composed of local cost at each node and the effect of the link is not incorporated. In fact, for consensus problems, though the decentralized algorithms may depend on the network topology (the links connecting nodes), the problem formulation itself is independent of the network structure. This is not suitable for many applications in distributed signal processing and control, where the notion of link cost or link utility naturally arises. For example, in multitask adaptive learning \cite{chen2014multitask}, each node $i$ aims at estimating its weight vector $\mathbf{w}_i$, which, in contrast to the consensus problems, is different from other nodes' weight vectors. In most networks, neighbor nodes tend to have similar weight vectors. To incorporate this prior knowledge into the estimator, the objective function to be minimized should include terms promoting similarity between neighbors such as $\|\mathbf{w}_i-\mathbf{w}_j\|_2^2$, where $i,j$ are neighbors. This term is tantamount to a link cost of the link $(i,j)$.

In this paper, we study the network cost minimization problem, where the network cost encompasses both node costs and link costs. To obtain a distributed algorithm for the problem, we resort to the distributed alternating direction method of multipliers (DADMM) \cite{boyd2011distributed}, which generally converges faster than distributed subgradient method \cite{nedic2009distributed}. However, each iteration of the DADMM algorithm involves solving a local optimization problem at each node, which is a major computational burden. To avoid this, we propose a distributed linearized ADMM (DLADMM) algorithm for network cost minimization. The DLADMM algorithm replaces the local optimization problem with closed form computations through linearizations and thus greatly reduce the computational complexity compared to DADMM. We further theoretically demonstrate that the DLADMM algorithm has appealing convergence properties. Our contributions can be summarized as follows.

\begin{itemize}
\item We formulate a generic form of network cost minimization problem incorporating both node costs and link costs. The formulated problem has broad applications in distributed signal processing and control in networked systems.
\item A distributed linearized ADMM algorithm for the network cost minimization problem is presented. The DLADMM algorithm operates in a decentralized manner and each iteration only consists of simple closed form computations, which endows the DLADMM with much lower computational overhead than the DADMM algorithm.
\item We prove that the DLADMM algorithm converges to an optimal point if the local cost functions are convex and have Lipschitz continuous gradients. Linear convergence rate of the DLADMM algorithm is also established provided that the local cost functions are further strongly convex.
\item Numerical experiments are conducted to validate the performance of the DLADMM algorithm. We empirically observe that the DLADMM algorithm has similar convergence speed as the DADMM algorithm does while the former enjoys much lower computational complexity. The impact of network topology, connectivity and algorithm parameters is also investigated.
\end{itemize}

The organization of the rest of this paper is as follows. In Section \Rmnum{2}, the network cost minimization problem is formally formulated and the DLADMM, DADMM algorithms are developed. In Section \Rmnum{3}, the convergence properties of the DLADMM algorithm are analyzed. In Section \Rmnum{4}, numerical simulations are conducted. In Section \Rmnum{5}, we conclude this work.

\section{Problem Statement and Algorithm Development}
In this section, we first motivate and formulate the network cost minimization problem. Then, we present a brief review of the basics of the ADMM, following which a distributed ADMM (DADMM) algorithm for the network cost minimization problem is shown. Finally, to reduce the computational burden of the DADMM, we propose a distributed linearized ADMM (DLADMM) algorithm for the network cost minimization problem.

\subsection{The Statement of the Problem}
Consider a network of $n$ nodes and some links between these nodes. We assume that the network is a simple graph, i.e., the network is undirected with no self-loop and there is at most one edge between any pair of nodes. Denote the number of links as $m$, in which $(i,j)$ and $(j,i)$ are counted as two links for ease of later exposition. Denote the set of neighbors of node $i$ (those who are linked with node $i$) as $\Omega_i$. The network can be either connected or disconnected (there does not necessarily exist a path connecting every pair of nodes). Suppose each node $i$ has a $p$-dimensional local decision variable $\mathbf{x}_i\in\mathbb{R}^p$. Given $\mathbf{x}_i$, the cost at node $i$ is $f_i(\mathbf{x}_i)$, where $f_i$ is called the node cost function at node $i$. Moreover, given two connected nodes $i$ and $j$ and their decision variables $\mathbf{x}_i$ and $\mathbf{x}_j$, there is a cost of $g_{ij}(\mathbf{x}_i,\mathbf{x}_j)$ associated with the link $(i,j)$, where $g_{ij}$ is called the link cost function of the link $(i,j)$. The goal of the network is to solve the following network cost minimization problem in a decentralized manner:
\begin{align}\label{NCM}
\text{Minimize}~\sum_{i=1}^nf_i(\mathbf{x}_i)+\sum_{i=1}^n\sum_{j\in\Omega_i}g_{ij}(\mathbf{x}_i,\mathbf{x}_j).
\end{align}
The problem formulation \eqref{NCM} has broad applications, among which we name three in the following.
\begin{itemize}
\item In distributed estimation over (sensor) networks, each node $i$ has a local unknown vector $\mathbf{x}_i$ to be estimated. The cost at node $i$, i.e., $f_i(\mathbf{x}_i)$ may be some squared error or the negative log-likelihood (the former can be regarded as a special case of the latter when the noise is Gaussian) with respect to the local data observed by node $i$. The link cost $g_{ij}(\mathbf{x}_i,\mathbf{x}_j)$ for a link $(i,j)$ can be used to enforce similarity between neighbor nodes, e.g., $\|\mathbf{x}_i-\mathbf{x}_j\|_2^2$ in multitask adaptive networks in \cite{chen2014multitask,monajemi2016informed}.
\item For resource allocation over networks, $\mathbf{x}_i$ corresponds to some resources at node $i$ and the node cost $f_i(\mathbf{x}_i)$ is the negative of node $i$'s utility. The link cost $g_{ij}(\mathbf{x}_i,\mathbf{x}_j)$ for a link $(i,j)$ may represent the negative effect of the consumption of the resources $\mathbf{x}_i$ and $\mathbf{x}_j$. For instance, in wireless networks, $\mathbf{x}_i$ may be the transmission power of node $i$ and two nodes are linked if they are within the wireless interference range. In such a case, the link cost $g_{ij}(\mathbf{x}_i,\mathbf{x}_j)$ for a link $(i,j)$ can be used to quantify the cost incurred by mutual interference in wireless communications.
\item For an image, each $x_i$ is the value of the $i$-th pixel and two pixels (or nodes) are linked if they are adjacent. In the image denoising problem, one wants to minimize the total variations of the pixels (as noises are often irregular values making the pixels abnormally different from their neighbor pixels) while remaining faithful to the given noisy image. The node cost $f_i(x_i)$ can be used to quantify the deviation of $x_i$ from the given noisy pixel $\widetilde{x}_i$ and the link cost $g_{ij}(x_i,x_j)$ can represent the difference between the two neighbor pixels $i$ and $j$.
\end{itemize}

For ease of reference, we define the following assumptions, some of which will be adopted in later theorems.

\begin{assump}
All the node cost functions $f_i$'s and the link cost functions $g_{ij}$'s are convex.
\end{assump}

\begin{assump}
All the node cost functions $f_i$'s and the link cost functions $g_{ij}$'s have Lipschitz continuous gradients with constant $L>0$, i.e., (a) $\forall i,\mathbf{x}_i,\mathbf{x}_i'\in\mathbb{R}^p$:
\begin{align}
\|\nabla f_i(\mathbf{x}_i)-\nabla f_i(\mathbf{x}_i')\|_2\leq L\|\mathbf{x}_i-\mathbf{x}_i'\|_2;
\end{align}
(b) $\forall i,j\in\Omega_i,\mathbf{x}_i,\mathbf{x}_j,\mathbf{x}_i',\mathbf{x}_j'\in\mathbb{R}^p$:
\begin{align}
\|\nabla g_{ij}(\mathbf{x}_i,\mathbf{x}_j)-\nabla g_{ij}(\mathbf{x}_i',\mathbf{x}_j')\|_2\leq L\left\|
\left[
\begin{array}{c}
\mathbf{x}_i\\
\mathbf{x}_j
\end{array}
\right]-
\left[
\begin{array}{c}
\mathbf{x}_i'\\
\mathbf{x}_j'
\end{array}
\right]
\right\|_2.
\end{align}
\end{assump}

\begin{assump}
All the node cost functions $f_i$'s and the link cost functions $g_{ij}$'s are strongly convex with constant $\tau>0$, i.e., (a) For any $i=1,...,n$:
\begin{align}\label{strong_f}
(\nabla f_i(\mathbf{x}_i)-\nabla f_i(\mathbf{x}_i'))^\mathsf{T}(\mathbf{x}_i-\mathbf{x}_i')\geq\tau\|\mathbf{x}_i-\mathbf{x}_i'\|_2^2,~~\forall \mathbf{x}_i,\mathbf{x}_i'\in\mathbb{R}^p;
\end{align}
(b)  For any $i,j\in\Omega_i$:
\begin{align}\label{strong_g}
\begin{split}
&\left(
\left[
\begin{array}{c}
\nabla_{\mathbf{x}_i}g_{ij}(\mathbf{x}_i,\mathbf{x}_j)\\
\nabla_{\mathbf{x}_j}g_{ij}(\mathbf{x}_i,\mathbf{x}_j)
\end{array}
\right]-
\left[
\begin{array}{c}
\nabla_{\mathbf{x}_i'}g_{ij}(\mathbf{x}_i',\mathbf{x}_j')\\
\nabla_{\mathbf{x}_j'}g_{ij}(\mathbf{x}_i',\mathbf{x}_j')
\end{array}
\right]
\right)^\mathsf{T}\left(\left[
\begin{array}{c}
\mathbf{x}_i\\
\mathbf{x}_j
\end{array}
\right]-
\left[
\begin{array}{c}
\mathbf{x}_i'\\
\mathbf{x}_j'
\end{array}
\right]\right)\\
&\geq\tau\left\|\left[
\begin{array}{c}
\mathbf{x}_i\\
\mathbf{x}_j
\end{array}
\right]-
\left[
\begin{array}{c}
\mathbf{x}_i'\\
\mathbf{x}_j'
\end{array}
\right]\right\|_2^2,~~\forall\mathbf{x}_i,\mathbf{x}_j,\mathbf{x}_i',\mathbf{x}_j'\in\mathbb{R}^p.
\end{split}
\end{align}
\end{assump}

\begin{rem}
We note the following facts. When $f_i$ is twice differentiable, the condition \eqref{strong_f} of Assumption 3 is equivalent to $\nabla^2f_i(\mathbf{x}_i)\succeq\tau\mathbf{I}_p,\forall\mathbf{x}_i$.
Similarly, when $g_{ij}$ is twice differentiable, the condition \eqref{strong_g} of Assumption 3 is equivalent to $\nabla^2g_{ij}(\mathbf{x}_i,\mathbf{x}_j)\succeq\tau\mathbf{I}_{2p},\forall\mathbf{x}_i,\mathbf{x}_j$. This second order definition of strong convexity is more intuitively acceptable and has been used in the analysis of convex optimization algorithms in the literature \cite{boyd2004convex}. But it requires twice differentiability and is not directly useful in the analysis in this work.
\end{rem}

\begin{rem}
All three assumptions are standard in the literature of numerical optimization when analyzing the performance of optimization algorithms \cite{boyd2004convex,shi2014linear,deng2016global}.
\end{rem}
\subsection{Preliminaries of ADMM}
ADMM is an optimization framework widely applied to various signal processing applications, including wireless communications \cite{shen2012distributed}, power systems \cite{zhang2016admm} and multi-agent coordination \cite{chang2014proximal}. It enjoys fast convergence speed under mild technical conditions \cite{deng2016global} and is especially suitable for the development of distributed algorithms \cite{boyd2011distributed,bertsekas1989parallel}. ADMM solves problems of the following form:
\begin{eqnarray}\label{admm_prime}
\text{Minimize}_{\mathbf{x},\mathbf{z}} f(\mathbf{x})+g(\mathbf{z})~~\text{s.t.}~~\mathbf{Ax+Bz=c},
\end{eqnarray}
where $\mathbf{A}\in\mathbb{R}^{p\times n},B\in\mathbb{R}^{p\times m},c\in\mathbb{R}^p$ are constants and $\mathbf{x}\in\mathbb{R}^n,\mathbf{z}\in\mathbb{R}^m$ are optimization variables. $f:\mathbb{R}^n\mapsto\mathbb{R}$ and $g:\mathbb{R}^m\mapsto\mathbb{R}$ are two convex functions. The augmented Lagrangian can be formed as:
\begin{equation}
\mathfrak{L}_\rho(\mathbf{x,z,y})=f(\mathbf{x})+g(\mathbf{z})+\mathbf{y}^\mathsf{T}(\mathbf{Ax+Bz-c})+\frac{\rho}{2}\|\mathbf{Ax+Bz-c}\|_2^2,
\end{equation}
where $\mathbf{y}\in\mathbb{R}^p$ is the Lagrange multiplier and $\rho>0$ is some constant. The ADMM then iterates over the following three steps for $k\geq0$ (the iteration index):
\begin{eqnarray}
&&\mathbf{x}^{k+1}=\arg\min_\mathbf{x} \mathfrak{L}_\rho\left(\mathbf{x},\mathbf{z}^k,\mathbf{y}^k\right),\label{x_prime}\\
&&\mathbf{z}^{k+1}=\arg\min_\mathbf{z} \mathfrak{L}_\rho\left(\mathbf{x}^{k+1},\mathbf{z},\mathbf{y}^k\right),\label{z_prime}\\
&&\mathbf{y}^{k+1}=\mathbf{y}^{k}+\rho\left(\mathbf{Ax}^{k+1}+\mathbf{Bz}^{k+1}-\mathbf{c}\right).\label{multiplier_prime}
\end{eqnarray}
The ADMM is guaranteed to converge to the optimal point of \eqref{admm_prime} as long as $f$ and $g$ are convex \cite{boyd2011distributed,bertsekas1989parallel}. It is recently shown that global linear convergence can be ensured provided additional assumptions on problem \eqref{admm_prime} holds \cite{deng2016global}.

\subsection{Development of the Distributed ADMM (DADMM) for Network Cost Minimization}
To develop an ADMM algorithm for \eqref{NCM}, we introduce auxiliary variables $\mathbf{y}_i$ and $\mathbf{z}_{ij}$ $\forall i,j\in\Omega_i$ and reformulate \eqref{NCM} equivalently as:
\begin{align}\label{admm_formulation}
\text{Minimize}~&\sum_{i=1}^nf_i(\mathbf{x}_i)+\sum_{i=1}^n\sum_{j\in\Omega_i}g_{ij}(\mathbf{y}_i,\mathbf{z}_{ij}).\\
\text{s.t.}~&\mathbf{x}_i=\mathbf{y}_i,~~i=1,...,n,\\
&\mathbf{x}_j=\mathbf{z}_{ij},~~i=1,...,n,j\in\Omega_i.
\end{align}

Further introducing Lagrangian multipliers $\boldsymbol{\lambda}_i,\boldsymbol{\mu}_{ij}\in\mathbb{R}^p,\forall i=1,...,n,j\in\Omega_i$, we form the augmented Lagrangian of the above optimization problem as:
\begin{align}
\begin{split}
\mathfrak{L}_\rho(\mathbf{x,y,z},\boldsymbol{\lambda,\mu})=&\sum_{i=1}^nf_i(\mathbf{x}_i)+\sum_{i=1}^n\sum_{j\in\Omega_i}g_{ij}(\mathbf{y}_i,\mathbf{z}_{ij})+\sum_{i=1}^n\boldsymbol{\lambda}_i^\mathsf{T}(\mathbf{x}_i-\mathbf{y}_i)+\sum_{i=1}^n\sum_{j\in\Omega_i}\boldsymbol{\mu}_{ij}^\mathsf{T}(\mathbf{x}_j-\mathbf{z}_{ij})\\
&+\frac{\rho}{2}\sum_{i=1}^n\|\mathbf{x}_i-\mathbf{y}_i\|_2^2+\frac{\rho}{2}\sum_{i=1}^n\sum_{j\in\Omega_i}\|\mathbf{x}_j-\mathbf{z}_{ij}\|_2^2,
\end{split}
\end{align}
where $\mathbf{x}\in\mathbb{R}^{np}$ is the concatenation of all $\mathbf{x}_i$'s into a column vector, i.e., $\mathbf{x}=\left[\mathbf{x}_1^\mathsf{T},...,\mathbf{x}_n^\mathsf{T}\right]^\mathsf{T}$; $\mathbf{y,}\boldsymbol{\lambda}\in\mathbb{R}^{np}$ are analogously defined; $\mathbf{z}\in\mathbb{R}^{mp}$ is the concatenation of all $\mathbf{z}_{ij}$'s in an arbitrary order of links; $\boldsymbol{\mu}\in\mathbb{R}^{mp}$ is analogously defined with the same link order as $\mathbf{z}$; $\rho>0$ is some positive constant. The ADMM algorithm can be derived as follows.

\subsubsection{Updating $\mathbf{x}$}
The update of $\mathbf{x}$ in the ADMM is:
\begin{align}
\mathbf{x}^{k+1}=\arg\min_\mathbf{x}\sum_{i=1}^nf_i(\mathbf{x}_i)+\sum_{i=1}^n\boldsymbol{\lambda}_i^{k\mathsf{T}}\mathbf{x}_i+\sum_{i=1}^n\sum_{j\in\Omega_i}\boldsymbol{\mu}_{ij}^{k\mathsf{T}}\mathbf{x}_j+\frac{\rho}{2}\sum_{i=1}^n\left\|\mathbf{x}_i-\mathbf{y}_i^k\right\|_2^2+\frac{\rho}{2}\sum_{i=1}^n\sum_{j\in\Omega_i}\left\|\mathbf{x}_j-\mathbf{z}_{ij}^k\right\|_2^2,
\end{align}
which can be decomposed across nodes: $\forall i,$
\begin{align}\label{x_admm}
\mathbf{x}_i^{k+1}=\arg\min_{\mathbf{x}_i}f_i(\mathbf{x}_i)+\boldsymbol{\lambda}_i^{k\mathsf{T}}\mathbf{x}_i+\sum_{l\in\Omega_i}\boldsymbol{\mu}_{li}^{k\mathsf{T}}\mathbf{x}_i+\frac{\rho}{2}\left\|\mathbf{x}_i-\mathbf{y}_i^k\right\|_2^2+\frac{\rho}{2}\sum_{l\in\Omega_i}\left\|\mathbf{x}_i-\mathbf{z}_{li}^k\right\|_2^2.
\end{align}

\subsubsection{Updating $\mathbf{y,z}$}
The update of $\mathbf{y,z}$ in the ADMM is:
\begin{align}
\begin{split}
\left\{\mathbf{y}^{k+1},\mathbf{z}^{k+1}\right\}=\arg\min_{\mathbf{y,z}}&\sum_{i=1}^n\sum_{j\in\Omega_i}g_{ij}(\mathbf{y}_i,\mathbf{z}_{ij})-\sum_{i=1}^n\boldsymbol{\lambda}_i^{k\mathsf{T}}\mathbf{y}_i-\sum_{i=1}^n\sum_{j\in\Omega_i}\boldsymbol{\mu}_{ij}^{k\mathsf{T}}\mathbf{z}_{ij}+\frac{\rho}{2}\sum_{i=1}^n\left\|\mathbf{y}_i-\mathbf{x}_i^{k+1}\right\|_2^2,\\
&+\frac{\rho}{2}\sum_{i=1}^n\sum_{j\in\Omega_i}\left\|\mathbf{z}_{ij}-\mathbf{x}_j^{k+1}\right\|_2^2
\end{split}
\end{align}
which can be decomposed across nodes: $\forall i,$
\begin{align}
\begin{split}\label{yz_admm}
\left\{\mathbf{y}_i^{k+1},\left\{\mathbf{z}_{ij}^{k+1}\right\}_{j\in\Omega_i}\right\}=\arg\min_{\mathbf{y}_i,\{\mathbf{z}_{ij}\}_{j\in\Omega_i}}&\sum_{j\in\Omega_i}g_{ij}(\mathbf{y}_i,\mathbf{z}_{ij})-\boldsymbol{\lambda}_i^{k\mathsf{T}}\mathbf{y}_i-\sum_{j\in\Omega_i}\boldsymbol{\mu}_{ij}^{k\mathsf{T}}\mathbf{z}_{ij}+\frac{\rho}{2}\left\|\mathbf{y}_i-\mathbf{x}_i^{k+1}\right\|_2^2\\
&+\frac{\rho}{2}\sum_{j\in\Omega_i}\left\|\mathbf{z}_{ij}-\mathbf{x}_j^{k+1}\right\|_2^2.
\end{split}
\end{align}

\subsubsection{Updating $\boldsymbol{\lambda},\boldsymbol{\mu}$}
The update of $\boldsymbol{\lambda},\boldsymbol{\mu}$ is also decomposed across nodes: $\forall i,j\in\Omega_i$
\begin{align}
\label{lambda_admm}&\boldsymbol{\lambda}_i^{k+1}=\boldsymbol{\lambda}_i^k+\rho\left(\mathbf{x}_i^{k+1}-\mathbf{y}_i^{k+1}\right),\\
\label{mu_admm}&\boldsymbol{\mu}_{ij}^{k+1}=\boldsymbol{\mu}_{ij}^k+\rho\left(\mathbf{x}_j^{k+1}-\mathbf{z}_{ij}^{k+1}\right).
\end{align}

Equations \eqref{x_admm}, \eqref{yz_admm}, \eqref{lambda_admm} and \eqref{mu_admm} together lead to a distributed ADMM (DADMM) algorithm for problem \eqref{NCM}, which is summarized from the perspective of an arbitrary node $i$ in Algorithm \ref{admm}. We note that only the values of $\mathbf{x,z},\boldsymbol{\mu}$ at the neighbors are needed for the ADMM updates. Therefore, in terms of information exchange, each node $i$ only needs to (i) broadcast $\mathbf{x}_i$ to the neighbors in $\Omega_i$; (ii) transmit $\mathbf{z}_{ij}$ to the neighbor $j$ for each $j\in\Omega_i$; (iii) transmit $\boldsymbol{\mu}_{ij}$ to the neighbor $j$ for each $j\in\Omega_i$.

\begin{algorithm}[!htbp]
\caption{The DADMM algorithm run at node $i$}
\begin{algorithmic}[1]\label{admm}
\STATE Initialize $\mathbf{x}_i^0=\mathbf{y}_i^0=\boldsymbol{\lambda}_i^0=\mathbf{0}$ and $\mathbf{z}_{ij}^0=\boldsymbol{\mu}_{ij}^0=\mathbf{0},\forall j\in\Omega_i$. $k=0$.
\STATE \textbf{Repeat:}
\STATE Compute $\mathbf{x}_i^{k+1}$ by solving the local optimization problem \eqref{x_admm} and then broadcast $\mathbf{x}_i^{k+1}$ to the neighbors $\Omega_i$.
\STATE Compute $\mathbf{y}_i^{k+1}$ and $\mathbf{z}_{ij}^{k+1},j\in\Omega_i$ by solving the local optimization problem \eqref{yz_admm} and then transmit $\mathbf{z}_{ij}^{k+1}$ to the neighbor node $j$ for each $j\in\Omega_i$.
\STATE Compute $\boldsymbol{\lambda}_i^{k+1}$ and $\boldsymbol{\mu}_{ij}^{k+1},j\in\Omega_i$ according to \eqref{lambda_admm} and \eqref{mu_admm}, respectively. Transmit $\boldsymbol{\mu}_{ij}^{k+1}$ to the neighbor node $j$ for each $j\in\Omega_i$.
\STATE $k\leftarrow k+1$.
\end{algorithmic}
\end{algorithm}

\subsection{Development of the Distributed Linearized ADMM (DLADMM) for Network Cost Minimization}

In the DADMM, i.e., Algorithm \ref{admm}, the updates for $\mathbf{x,y,z}$ involve solving local optimization problems \eqref{x_admm} and \eqref{yz_admm}, which generally do not admit close-form solutions and have to be solved iteratively. This can be a major computational burden for Algorithm \ref{admm} especially when individual node has only limited computational capability, e.g., the cheap sensors vastly deployed in sensor networks usually can only carry out simple calculations. This motivates us to propose an algorithm which can approximately solve the local optimization problems efficiently and most preferably with closed form solutions. To this end, we first define $f(\mathbf{x})=\sum_{i=1}^nf_i(\mathbf{x}_i)$ and $g(\mathbf{y},\mathbf{z})=\sum_{i=1}^n\sum_{j\in\Omega_i}g_{ij}(\mathbf{y}_i,\mathbf{z}_{ij})$. We further define a block matrix $\mathbf{A}\in\mathbb{R}^{mp\times np}$ consisting of $m\times n$ blocks of matrices $\mathbf{A}_{kj}\in\mathbb{R}^{p\times p}$, where $\mathbf{A}_{kj}$ is equal to $\mathbb{I}_{p\times p}$ if the $k$-th $p$-dimensional block of $\mathbf{z}$ is $\mathbf{z}_{ij}$ for some $i=1,...,n$, otherwise $\mathbf{A}_{kj}$ is equal to $\mathbf{0}_{p\times p}$. Then, we may rewrite problem \eqref{admm_formulation} compactly as:
\begin{align}\label{admm_formulation_1}
\text{Minimize}~~&f(\mathbf{x})+g(\mathbf{y,z})\\
\text{s.t.}~~&\mathbf{x=y},\\
&\mathbf{Ax=z}.
\end{align}
Further define $\mathbf{w}=\left[\mathbf{y}^\mathsf{T},\mathbf{z}^\mathsf{T}\right]^\mathsf{T}$ and $\mathbf{B}=\left[\mathbf{I},\mathbf{A}^\mathsf{T}\right]^\mathsf{T}$. Thus, \eqref{admm_formulation_1} can be rewritten as:
\begin{align}\label{admm_formulation_2}
\text{Minimize}~~&f(\mathbf{x})+g(\mathbf{w})\\
\text{s.t.}~~&\mathbf{Bx-w=0}.
\end{align}
The augmented Lagrangian can be written as:
\begin{align}
\mathfrak{L}_\rho(\mathbf{x,w,}\boldsymbol{\alpha})=f(\mathbf{x})+g(\mathbf{w})+\alpha^\mathsf{T}(\mathbf{Bx-w})+\frac{\rho}{2}\|\mathbf{Bx-w}\|_2^2,
\end{align}
where $\boldsymbol\alpha=\left[\boldsymbol\lambda^\mathsf{T},\boldsymbol\mu^\mathsf{T}\right]^\mathsf{T}$ is the Lagrangian multiplier. The original DADMM algorithm necessitates solving local optimization problems involving $f$ and $g$. To avoid this burden, we approximate $f,g$ with their first order approximations and propose a distributed linearized ADMM (DLADMM) algorithm for network cost minimization in the following.

\subsubsection{Updating $\mathbf{x}$}
The update of $\mathbf{x}$ in DLADMM is:
\begin{align}\label{x_DLADMM_opt}
\mathbf{x}^{k+1}=\arg\min_\mathbf{x}\nabla f\left(\mathbf{x}^k\right)^\mathsf{T}\left(\mathbf{x}-\mathbf{x}^k\right)+\frac{c}{2}\left\|\mathbf{x}-\mathbf{x}^k\right\|_2^2+\boldsymbol{\alpha}^{k\mathsf{T}}\mathbf{Bx}+\frac{\rho}{2}\left\|\mathbf{Bx-w}^k\right\|_2^2,
\end{align}
where $c>0$ is some positive constant and the term $\frac{c}{2}\left\|\mathbf{x}-\mathbf{x}^k\right\|_2^2$ is to refrain $\mathbf{x}^{k+1}$ from being too far away from $\mathbf{x}^k$ as the first order approximation of $f$ around the point $\mathbf{x}^k$ is only accurate when $\mathbf{x}$ is close to $\mathbf{x}^k$. Note that this small step size or small variation between iterations is common in the literature of numerical optimization \cite{boyd2004convex} and adaptive signal processing such as least mean squares (LMS) \cite{Haykin:1996:AFT:230061}. Since the objective function in \eqref{x_DLADMM_opt} is a convex quadratic function of $\mathbf{x}$, the problem of \eqref{x_DLADMM_opt} can be solved in closed form through the first order condition:
\begin{align}\label{x_dladmm_c}
\nabla f\left(\mathbf{x}^k\right)+c\left(\mathbf{x}^{k+1}-\mathbf{x}^k\right)+\mathbf{B}^\mathsf{T}\boldsymbol{\alpha}^k+\rho\left(\mathbf{B}^\mathsf{T}\mathbf{Bx}^{k+1}-\mathbf{B}^\mathsf{T}\mathbf{w}^k\right)=0.
\end{align}
We note that the optimization problem \eqref{x_DLADMM_opt} can be decomposed across nodes:
\begin{align}
\begin{split}
\mathbf{x}_i^{k+1}=\arg\min_{\mathbf{x}_i}&\nabla f_i\left(\mathbf{x}_i^k\right)^\mathsf{T}\left(\mathbf{x}_i-\mathbf{x}_i^k\right)+\frac{c}{2}\left\|\mathbf{x}_i-\mathbf{x}_i^k\right\|_2^2+\boldsymbol{\lambda}_i^{k\mathsf{T}}\mathbf{x}_i+\sum_{l\in\Omega_i}\boldsymbol{\mu}_{li}^{k\mathsf{T}}\mathbf{x}_i+\frac{\rho}{2}\left\|\mathbf{x}_i-\mathbf{y}_i^k\right\|_2^2\\
&+\frac{\rho}{2}\sum_{l\in\Omega_i}\left\|\mathbf{x}_i-\mathbf{z}_{li}^k\right\|_2^2,
\end{split}
\end{align}
which can be solved in closed form:
\begin{align}\label{x_dladmm}
\mathbf{x}_i^{k+1}=\frac{1}{c+\rho+\rho|\Omega_i|}\left[-\nabla f_i\left(\mathbf{x}_i^k\right)+c\mathbf{x}_i^k-\lambda_i^k-\sum_{l\in\Omega_i}\boldsymbol{\mu}_{li}^k+\rho\mathbf{y}_i^k+\rho\sum_{l\in\Omega_i}\mathbf{z}_{li}^k\right].
\end{align}

\subsubsection{Updating $\mathbf{w}$, i.e., $\mathbf{y}$ and $\mathbf{z}$}
The update of $\mathbf{w}$ in the DLADMM algorithm is:
\begin{align}\label{w_DLADMM_opt}
\mathbf{w}^{k+1}=\arg\min_\mathbf{w}\nabla g\left(\mathbf{w}^k\right)^\mathsf{T}\left(\mathbf{w-w}^k\right)+\frac{c}{2}\left\|\mathbf{w-w}^k\right\|_2^2-\boldsymbol{\alpha}^{k\mathsf{T}}\mathbf{w}+\frac{\rho}{2}\left\|\mathbf{w-Bx}^{k+1}\right\|_2^2,
\end{align}
which is equivalent to:
\begin{align}\label{w_dladmm_c}
\nabla g(\mathbf{w}^k)+c\left(\mathbf{w}^{k+1}-\mathbf{w}^k\right)-\boldsymbol{\alpha}^k+\rho\left(\mathbf{w}^{k+1}-\mathbf{Bx}^{k+1}\right)=0.
\end{align}
Notice that the problem \eqref{w_DLADMM_opt} can also be decomposed across nodes:
\begin{align}
\begin{split}
\left\{\mathbf{y}_i^{k+1},\left\{\mathbf{z}_{ij}^{k+1}\right\}_{j\in\Omega_i}\right\}=\arg\min_{\mathbf{y}_i,\{\mathbf{z}_{ij}\}_{j\in\Omega_i}}&\sum_{j\in\Omega_i}
\left[
\begin{array}{c}
\nabla_{\mathbf{y}_i}g_{ij}\left(\mathbf{y}_i^k,\mathbf{z}_{ij}^k\right)\\
\nabla_{\mathbf{z}_{ij}}g_{ij}\left(\mathbf{y}_i^k,\mathbf{z}_{ij}^k\right)
\end{array}
\right]^\mathsf{T}
\left[
\begin{array}{c}
\mathbf{y}_i-\mathbf{y}_i^k\\
\mathbf{z}_{ij}-\mathbf{z}_{ij}^k
\end{array}
\right]
+\frac{c}{2}\left\|\mathbf{y}_i-\mathbf{y}_i^k\right\|_2^2\\
&+\frac{c}{2}\sum_{j\in\Omega_i}\left\|\mathbf{z}_{ij}-\mathbf{z}_{ij}^k\right\|_2^2-\boldsymbol{\lambda}_i^{k\mathsf{T}}\mathbf{y}_i-\sum_{j\in\Omega_i}\boldsymbol{\mu}_{ij}^{k\mathsf{T}}\mathbf{z}_{ij}\\
&+\frac{\rho}{2}\left\|\mathbf{y}_i-\mathbf{x}_i^{k+1}\right\|_2^2+\frac{\rho}{2}\sum_{j\in\Omega_i}\left\|\mathbf{z}_{ij}-\mathbf{x}_j^{k+1}\right\|_2^2,
\end{split}
\end{align}
which can be solved as:
\begin{align}
\label{y_dladmm}&\mathbf{y}_i^{k+1}=\frac{1}{c+\rho}\left[-\sum_{j\in\Omega_i}\nabla_{\mathbf{y}_i}g_{ij}\left(\mathbf{y}_i^k,\mathbf{z}_{ij}^k\right)+c\mathbf{y}_i^k+\boldsymbol{\lambda}_i^k+\rho\mathbf{x}_i^{k+1}\right],\\
\label{z_dladmm}&\mathbf{z}_{ij}^{k+1}=\frac{1}{c+\rho}\left[-\nabla_{\mathbf{z}_{ij}}g_{ij}\left(\mathbf{y}_i^k,\mathbf{z}_{ij}^k\right)+c\mathbf{z}_{ij}^k+\boldsymbol{\mu}_{ij}^k+\rho\mathbf{x}_j^{k+1}\right].
\end{align}

\subsubsection{Updating $\boldsymbol{\alpha}$, i.e., $\boldsymbol{\lambda}$ and $\boldsymbol{\mu}$}
The update of $\boldsymbol{\alpha}$ is:
\begin{align}\label{alpha_dladmm_c}
\boldsymbol{\alpha}^{k+1}=\boldsymbol{\alpha}^k+\rho\left(\mathbf{Bx}^{k+1}-\mathbf{w}^{k+1}\right),
\end{align}
which can be implemented in a decentralized manner as in \eqref{lambda_admm} and \eqref{mu_admm}. In other words, the update of the dual variables in DLADMM is the same as that of DADMM. Combining \eqref{x_dladmm}, \eqref{y_dladmm}, \eqref{z_dladmm}, \eqref{lambda_admm} and \eqref{mu_admm} yields the proposed DLADMM algorithm, which is summarized in Algorithm \ref{dladmm}. We remark that, as opposed to Algorithm \ref{admm}, each iteration of Algorithm \ref{dladmm} only involves direct closed form computations without solving any local optimization problems iteratively. This enables DLADMM to enjoy significantly lower computational complexity compared to DADMM.

\begin{algorithm}[!htbp]
\caption{The DLADMM algorithm run at node $i$}
\begin{algorithmic}[1]\label{dladmm}
\STATE Initialize $\mathbf{x}_i^0=\mathbf{y}_i^0=\boldsymbol{\lambda}_i^0=\mathbf{0}$ and $\mathbf{z}_{ij}^0=\boldsymbol{\mu}_{ij}^0=\mathbf{0},\forall j\in\Omega_i$. $k=0$.
\STATE \textbf{Repeat:}
\STATE Compute $\mathbf{x}_i^{k+1}$ according to \eqref{x_dladmm} and then broadcast $\mathbf{x}_i^{k+1}$ to the neighbors $\Omega_i$.
\STATE Compute $\mathbf{y}_i^{k+1}$ and $\mathbf{z}_{ij}^{k+1},j\in\Omega_i$ according to \eqref{y_dladmm} and \eqref{z_dladmm}, respectively. Then transmit $\mathbf{z}_{ij}^{k+1}$ to the neighbor node $j$ for each $j\in\Omega_i$.
\STATE Compute $\boldsymbol{\lambda}_i^{k+1}$ and $\boldsymbol{\mu}_{ij}^{k+1},j\in\Omega_i$ according to \eqref{lambda_admm} and \eqref{mu_admm}, respectively. Transmit $\boldsymbol{\mu}_{ij}^{k+1}$ to the neighbor node $j$ for each $j\in\Omega_i$.
\STATE $k\leftarrow k+1$.
\end{algorithmic}
\end{algorithm}

\section{Convergence Analysis}

In this section, we analyze the convergence behaviors of the proposed DLADMM algorithm for the network cost minimization problem \eqref{NCM}. Instead of analyzing the DLADMM algorithm outlined in Algorithm \ref{dladmm}, we will analyze its centralized version in \eqref{x_dladmm_c}, \eqref{w_dladmm_c} and \eqref{alpha_dladmm_c}, which are tantamount to their decentralized counterpart in Algorithm \ref{dladmm}. We perform convergence analysis based on \eqref{x_dladmm_c}, \eqref{w_dladmm_c} and \eqref{alpha_dladmm_c} as they are more compact and thus more amenable to analyses and expositions.\footnote{Note that completely equivalent analysis based on the decentralized implementation in Algorithm \ref{dladmm} can be conducted, though the notations are more cluttered.} Before formally analyzing the convergence of DLADMM, we first present some preliminaries. After that, we show the convergence guarantee of the DLADMM (Theorem 1) and in particular, a linear convergence rate of the DLADMM (Theorem 2).

\subsection{Preliminaries}

\begin{lem}\label{lemA}
If Assumption 2 holds, then $\nabla f$ is Lipschitz continuous with constant $L$ and $\nabla g$ is Lipschitz continuous with constant $M=\sqrt{L^2K^2+L^2K}$, where $K=\max_i|\Omega_i|$ is the maximum degree of the network.
\end{lem}
\begin{proof}
The proof is given in Appendix A.
\end{proof}
We further note that the following fact from convex analysis \cite{vandenberghe2016gradient}.
\begin{lem}\label{lemB}
For any differentiable convex function $h:\mathbb{R}^l\mapsto\mathbb{R}$ and positive constant $L>0$, the following two statements are equivalent:
\begin{enumerate}
\item $\nabla h$ is Lipschitz continuous with constant $L$, i.e., $\left\|\nabla h(\mathbf{x})-\nabla h\left(\mathbf{x}'\right)\right\|_2\leq L\left\|\mathbf{x}-\mathbf{x}'\right\|_2,\forall \mathbf{x},\mathbf{x}'\in\mathbb{R}^l$.
\item $\left\|\nabla h(\mathbf{x})-\nabla h\left(\mathbf{x}'\right)\right\|_2^2\leq L\left(\mathbf{x}-\mathbf{x}'\right)^\mathsf{T}\left(\nabla h(\mathbf{x})-\nabla h\left(\mathbf{x}'\right)\right),\forall\mathbf{x},\mathbf{x}'\in\mathbb{R}^l$.
\end{enumerate}
\end{lem}
Utilizing the Lemma \ref{lemA} and Lemma \ref{lemB}, we immediately have the following result.
\begin{lem}\label{lemE}
If Assumptions 1 and 2 hold, we have:
\begin{align}
&\left\|\nabla f(\mathbf{x})-\nabla f\left(\mathbf{x}'\right)\right\|_2^2\leq L\left(\mathbf{x}-\mathbf{x}'\right)^\mathsf{T}\left(\nabla f(\mathbf{x})-\nabla f\left(\mathbf{x}'\right)\right),\forall\mathbf{x},\mathbf{x}'\in\mathbb{R}^{np},\\
&\left\|\nabla g(\mathbf{w})-\nabla g\left(\mathbf{w}'\right)\right\|_2^2\leq M\left(\mathbf{w}-\mathbf{w}'\right)^\mathsf{T}\left(\nabla g(\mathbf{w})-\nabla g\left(\mathbf{w}'\right)\right),\forall\mathbf{w},\mathbf{w}'\in\mathbb{R}^{np+mp},
\end{align}
where $M$ is defined in Lemma \ref{lemA}.
\end{lem}

\subsection{Convergence}

We define a diagonal positive definite matrix $\mathbf{\Lambda}$ as:
\begin{align}
\mathbf{\Lambda}=
\left[
\begin{array}{ccc}
\frac{c}{2}\mathbf{I}_{np} & &\\
&\frac{\rho+c}{2}\mathbf{I}_{np+mp}&\\
&&\frac{1}{2\rho}\mathbf{I}_{np+mp}
\end{array}
\right].
\end{align}
For ease of notation, we define $\mathbf{u}\in\mathbb{R}^{3np+2mp}$ to be the concatenation of $\mathbf{x,w,}\boldsymbol{\alpha}$ into a single column vector and similarly for $\mathbf{u}^k,\mathbf{u}^*$. Since $\mathbf{\Lambda}$ is a positive definite matrix, we can further define a norm on $\mathbb{R}^{3np+2mp}$ as: $\|\mathbf{u}\|_\mathbf{\Lambda}=\sqrt{\mathbf{u}^\mathsf{T}\mathbf{\Lambda u}}$. We have the following result.
\begin{prop}
Suppose Assumptions 1 and 2 hold. Then, for any primal/dual optimal point of problem \eqref{admm_formulation_2} $\mathbf{u}^*=\left[\mathbf{x}^{*\mathsf{T}},\mathbf{w}^{*\mathsf{T}},\boldsymbol{\alpha}^{*\mathsf{T}}\right]^\mathsf{T}$, the sequence $\mathbf{u}^k=\left[\mathbf{x}^{k\mathsf{T}},\mathbf{w}^{k\mathsf{T}},\boldsymbol{\alpha}^{k\mathsf{T}}\right]^\mathsf{T}$ generated by the DLADMM algorithm satisfies $\forall k\geq0$:
\begin{align}
\begin{split}
\left\|\mathbf{u}^{k+1}-\mathbf{u}^*\right\|_\mathbf{\Lambda}^2\leq&\left\|\mathbf{u}^k-\mathbf{u}^*\right\|_\mathbf{\Lambda}^2-\left(\frac{c}{2}-\frac{L}{4}\right)\left\|\mathbf{x}^k-\mathbf{x}^{k+1}\right\|_2^2-\left(\frac{c-\rho}{2}-\frac{M}{4}\right)\left\|\mathbf{w}^k-\mathbf{w}^{k+1}\right\|_2^2\\
&-\frac{1}{4\rho}\left\|\boldsymbol{\alpha}^{k+1}-\boldsymbol{\alpha}^k\right\|_2^2.
\end{split}
\end{align}
\end{prop}
\begin{proof}
The proof is given in Appendix B.
\end{proof}

Now, we are ready to state our first main theorem of convergence.
\begin{thm}
Suppose Assumptions 1,2 hold and $c>\frac{M}{2}+\rho$. Then, the sequence $\mathbf{u}^k$ generated by the DLADMM algorithm converges to some primal/dual optimal point of problem \eqref{admm_formulation_2}, i.e., there exists a primal/dual optimal point of problem \eqref{admm_formulation_2} $\mathbf{u}^*$ such that $\lim_{k\rightarrow\infty}\mathbf{u}^k=\mathbf{u}^*$.
\end{thm}
\begin{proof}
Given any primal/dual optimal point of problem \eqref{admm_formulation_2} $\mathbf{u}^*$, according to Proposition 1 and $c>\frac{M}{2}+\rho$, we know that $\left\|\mathbf{u}^k-\mathbf{u}^*\right\|_\mathbf{\Lambda}^2$ is a decreasing sequence. Since it is clearly lower bounded by 0, we have that $\left\|\mathbf{u}^k-\mathbf{u}^*\right\|_\mathbf{\Lambda}^2$ is convergent. From Proposition 1, we further deduce that:
\begin{align}
\label{5}0&\leq\left(\frac{c}{2}-\frac{L}{4}\right)\left\|\mathbf{x}^k-\mathbf{x}^{k+1}\right\|_2^2+\left(\frac{c-\rho}{2}-\frac{M}{4}\right)\left\|\mathbf{w}^k-\mathbf{w}^{k+1}\right\|_2^2+\frac{1}{4\rho}\left\|\boldsymbol{\alpha}^{k+1}-\boldsymbol{\alpha}^k\right\|_2^2\\
&\leq\left\|\mathbf{u}^k-\mathbf{u}^*\right\|_\mathbf{\Lambda}^2-\left\|\mathbf{u}^{k+1}-\mathbf{u}^*\right\|_\mathbf{\Lambda}^2.\label{6}
\end{align}
Because $\left\|\mathbf{u}^k-\mathbf{u}^*\right\|_\mathbf{\Lambda}^2$ is convergent, we know \eqref{6} converges to zero as $k$ goes to infinity. Hence, \eqref{5} converges to zero as well. Therefore,
\begin{align}
&\lim_{k\rightarrow\infty}\left(\mathbf{x}^k-\mathbf{x}^{k+1}\right)=\mathbf{0},\\
\label{www}&\lim_{k\rightarrow\infty}\left(\mathbf{w}^k-\mathbf{w}^{k+1}\right)=\mathbf{0},\\
&\lim_{k\rightarrow\infty}\left(\boldsymbol{\alpha}^k-\boldsymbol{\alpha}^{k+1}\right)=\mathbf{0}.
\end{align}
Substituting the above limits into \eqref{x_dladmm_c}, \eqref{w_dladmm_c} and \eqref{alpha_dladmm_c} yields:
\begin{align}
\label{b1}&\lim_{k\rightarrow\infty}\left[\nabla f\left(\mathbf{x}^k\right)+\mathbf{B}^\mathsf{T}\boldsymbol{\alpha}^k+\rho\left(\mathbf{B}^\mathsf{T}\mathbf{Bx}^{k+1}-\mathbf{B}^\mathsf{T}\mathbf{w}^k\right)\right]=\mathbf{0},\\
\label{b2}&\lim_{k\rightarrow\infty}\left[\nabla g\left(\mathbf{w}^k\right)-\boldsymbol{\alpha}^k+\rho\left(\mathbf{w}^{k+1}-\mathbf{Bx}^{k+1}\right)\right]=\mathbf{0},\\
\label{b3}&\lim_{k\rightarrow\infty}\left(\mathbf{Bx}^{k+1}-\mathbf{w}^{k+1}\right)=\mathbf{0}.
\end{align}
Equation \eqref{b3} clearly implies:
\begin{align}\label{d1}
\lim_{k\rightarrow\infty}\left(\mathbf{Bx}^k-\mathbf{w}^k\right)=\mathbf{0}.
\end{align}
Combining \eqref{b2} and \eqref{b3} leads to:
\begin{align}\label{d2}
\lim_{k\rightarrow\infty}\left[\nabla g\left(\mathbf{w}^k\right)-\boldsymbol{\alpha}^k\right]=\mathbf{0}.
\end{align}
Moreover, from \eqref{b3} and \eqref{www}, we obtain:
\begin{align}
\label{7}\mathbf{Bx}^{k+1}-\mathbf{w}^k=\mathbf{Bx}^{k+1}-\mathbf{w}^{k+1}+\mathbf{w}^{k+1}-\mathbf{w}^k\rightarrow\mathbf{0},\text{as}~k\rightarrow\infty.
\end{align}
Combining \eqref{7} and \eqref{b1}, we get:
\begin{align}\label{d3}
\lim_{k\rightarrow\infty}\left[\nabla f\left(\mathbf{x}^k\right)+\mathbf{B}^\mathsf{T}\boldsymbol{\alpha}^k\right]=\mathbf{0}.
\end{align}
Since $\forall k:\left\|\mathbf{u}^k-\mathbf{u}^*\right\|_\mathbf{\Lambda}\leq\left\|\mathbf{u}^0-\mathbf{u}^*\right\|_\mathbf{\Lambda}$, we know that $\left\{\mathbf{u}^k\right\}_{k=0,1,...}$ is a bounded sequence. So, it has convergent subsequence, which is denoted as $\left\{\mathbf{u}^{k_i}\right\}_{i=1,2,...}$. Let $\widehat{\mathbf{u}}$ be the limit of this convergent subsequence, i.e., $\lim_{i\rightarrow\infty}\mathbf{u}^{k_i}=\widehat{\mathbf{u}}$. Equations \eqref{d1}, \eqref{d2} and \eqref{d3} are still satisfied along the subsequence $\left\{\mathbf{u}^{k_i}\right\}_{i=1,2,...}$ and hence,
\begin{align}
&\lim_{i\rightarrow\infty}\left(\mathbf{Bx}^{k_i}-\mathbf{w}^{k_i}\right)=\mathbf{0},\\
&\lim_{i\rightarrow\infty}\left[\nabla g\left(\mathbf{w}^{k_i}\right)-\boldsymbol{\alpha}^{k_i}\right]=\mathbf{0},\\
&\lim_{i\rightarrow\infty}\left[\nabla f\left(\mathbf{x}^{k_i}\right)+\mathbf{B}^\mathsf{T}\boldsymbol{\alpha}^{k_i}\right]=\mathbf{0}.
\end{align}
Making use of the convergence of the subsequence $\left\{\mathbf{u}^{k_i}\right\}_{i=1,2,...}$ to $\widehat{\mathbf{u}}$, we obtain:
\begin{align}
&\mathbf{B}\widehat{\mathbf{x}}-\widehat{\mathbf{w}}=\mathbf{0},\\
&\nabla g\left(\widehat{\mathbf{w}}\right)-\widehat{\boldsymbol{\alpha}}=\mathbf{0},\\
&\nabla f\left(\widehat{\mathbf{x}}\right)+\mathbf{B}^\mathsf{T}\widehat{\boldsymbol{\alpha}}=\mathbf{0}.
\end{align}
These are the KKT conditions of problem \eqref{admm_formulation_2}. So $\widehat{\mathbf{u}}$ is a primal/dual optimal point of problem \eqref{admm_formulation_2}. In the following, we endeavor to show that the sequence $\mathbf{u}^k$ converges to $\widehat{\mathbf{u}}$. Before that, we first prove a lemma.
\begin{lem}\label{lemC}
If the sequence $\left\{\mathbf{u}^k\right\}_{k=0,1,...}$ has two subsequences $\left\{\mathbf{u}^{k_i}\right\}_{i=1,2,...}$ and $\left\{\mathbf{u}^{k_i'}\right\}_{i=1,2,...}$ converging to $\underline{\mathbf{u}}$ and $\overline{\mathbf{u}}$, respectively, then $\underline{\mathbf{u}}=\overline{\mathbf{u}}$.
\end{lem}
\begin{proof}
The proof is given in Appendix C.
\end{proof}

Now, we show that $\mathbf{u}^k$ converges to $\widehat{\mathbf{u}}$ by making use of Lemma \ref{lemC}. Suppose, on the contrary, $\mathbf{u}^k$ does not converge to $\widehat{\mathbf{u}}$. Then, there exists some positive $\epsilon$, such that for any positive integer $N$, there exists some $k\geq N$ with $\left\|\mathbf{u}^k-\widehat{\mathbf{u}}\right\|_2\geq\epsilon$. Thus, letting $N=1$, we get some $\tilde{k}_1\geq 1$ with $\left\|\mathbf{u}^{\tilde{k}_1}-\widehat{\mathbf{u}}\right\|_2\geq\epsilon$. Letting $N=\tilde{k}_1+1$, we get some $\tilde{k}_2\geq \tilde{k}_1+1$ with $\left\|\mathbf{u}^{\tilde{k}_2}-\widehat{\mathbf{u}}\right\|_2\geq\epsilon$. Continuing this process, we obtain a subsequence $\left\{\mathbf{u}^{\tilde{k}_i}\right\}_{i=1,2,...}$ such that $\left\|\mathbf{u}^{\tilde{k}_i}-\widehat{\mathbf{u}}\right\|_2\geq\epsilon,\forall i$. The subsequence $\left\{\mathbf{u}^{\tilde{k}_i}\right\}_{i=1,2,...}$ is bounded as the original sequence $\left\{\mathbf{u}^k\right\}_{k=0,1,...}$ is bounded. As such, the subsequence $\left\{\mathbf{u}^{\tilde{k}_i}\right\}_{i=1,2,...}$ has a convergent sub-subsequence $\left\{\mathbf{u}^{\tilde{k}_{i_j}}\right\}_{j=1,2,...}$. Denote the limit of this convergent sub-subsequence as $\widetilde{\mathbf{u}}$, i.e., $\lim_{j\rightarrow\infty}\mathbf{u}^{\tilde{k}_{i_j}}=\widetilde{\mathbf{u}}$. Obviously, $\left\|\widetilde{\mathbf{u}}-\widehat{\mathbf{u}}\right\|_2\geq\epsilon$. But, according to Lemma \ref{lemC}, we should have $\widetilde{\mathbf{u}}=\widehat{\mathbf{u}}$. This is a contradiction. So, we must have $\lim_{k\rightarrow\infty}\mathbf{u}^k=\widehat{\mathbf{u}}$. Note that we have previously shown that $\widehat{\mathbf{u}}$ is a primal/dual optimal point of problem \eqref{admm_formulation_2}. We hence conclude the theorem.
\end{proof}

\subsection{Linear Rate of Convergence}
With the strong convexity assumption, we can further guarantee linear convergence rate of the DLADMM algorithm. Before formally stating this result, we first show an implication of Assumption 3.
\begin{lem}\label{lemD}
If Assumption 3 holds, then $f$ and $g$ are both strongly convex with constant $\tau$.
\end{lem}
\begin{proof}
The proof is given in Appendix D.
\end{proof}

The strong convexity of $f$ and $g$ implies that there exists a unique primal/dual optimal point $\mathbf{u}^*$ for problem \eqref{admm_formulation_2}. Denote the spectral norm (maximum singular value) of $\mathbf{B}$ as $\Gamma$. Now, we are ready to state our second main theorem regarding linear convergence rate.

\begin{thm}
Suppose Assumptions 2,3 hold and $c>\max\left\{\frac{L^2}{2\tau},\rho+\frac{M^2}{2\tau}\right\}$. Then, $\forall k$:
\begin{align}
\left\|\mathbf{u}^{k+1}-\mathbf{u}^*\right\|_\mathbf{\Lambda}^2\leq\frac{1}{1+\delta}\left\|\mathbf{u}^k-\mathbf{u}^*\right\|_\mathbf{\Lambda}^2.\label{linear_convergence}
\end{align}
In \eqref{linear_convergence}, $\delta>0$ is a positive constant defined as:
\begin{align}\label{delta_def}
\delta=\min\left\{\frac{\tau-\frac{L^2}{2c}}{\frac{c}{2}+\frac{3\rho\mu\Gamma^2}{\mu-1}},\frac{\tau-\frac{\beta}{2}}{\frac{c+\rho}{2}+\frac{3\rho\mu}{\mu-1}+\frac{2M^2\mu}{\rho}},\frac{1}{4}\right\},
\end{align}
where $\beta\in\left(\frac{M^2}{c-\rho},2\tau\right)$ is the solution of the equation:
\begin{align}\label{beta_def}
\frac{\tau-\frac{\beta}{2}}{\frac{c+\rho}{2}+\frac{3\rho\mu}{\mu-1}+\frac{2M^2\mu}{\rho}}=\frac{\frac{c-\rho}{2}-\frac{M^2}{2\beta}}{\frac{3c^2\mu}{\rho(\mu-1)}+\frac{2M^2\mu}{\rho}},
\end{align}
and $\mu>1$ is any constant greater than 1.
\end{thm}
\begin{proof}
The proof is in Appendix E.

\end{proof}

\begin{rem}
The constant $\delta$ determining the convergence rate of the DLADMM depends on the local cost functions ($L,M,\tau$), the network topology ($\Gamma$) as well as the algorithm parameters ($\rho,c$). This sheds some light on how to tune the parameters to achieve better convergence speed in practice. Furthermore, we note that the Theorem 2 only provides a sufficient condition for linear convergence of the DLADMM. In later numerical experiments, we will see that even when the assumptions of Theorem 2 is violated (e.g., the local cost functions are not strongly convex), the DLADMM algorithm may still converge in linear rate.
\end{rem}

\section{Numerical Experiments}

In this section, numerical results are presented to corroborate the effectiveness of the proposed DLADMM algorithm. In particular, we consider the problem of distributed logistic regression. Suppose each node $i$ has a training set of $q$ training examples $\{\mathbf{u}_{il},t_{il}\}_{l=1,...,q}$, where $\mathbf{u}_{il}\in\mathbb{R}^p$ is the input feature vector and $t_{il}\in\{-1,1\}$ is the corresponding output label. Logistic regression model postulates that, for node $i$, the probability of the output $t_i$ given the input $\mathbf{u}_i$ is $\Pr(t_i=1|\mathbf{u}_i)=\frac{1}{1+\exp\left\{-\mathbf{u}_i^\mathsf{T}\mathbf{x}_i\right\}}$, where $\mathbf{x}_i$ is the classifier for node $i$. Our goal is to estimate the classifiers of all nodes and thus, together with a decision threshold, we can achieve a input-output mapping at each node. Moreover, we note that neighbor nodes tend to have similar classifiers. Incorporating this prior knowledge into the maximum likelihood estimator of the logistic regression yields the following optimization problem:
\begin{align}
\text{Minimize}_{\{\mathbf{x}_i\}_{i=1,...,n}}~~\sum_{i=1}^n\sum_{l=1}^q\log\left(1+\exp\left(-t_{il}\mathbf{u}_{il}^\mathsf{T}\mathbf{x}_i\right)\right)+\beta\sum_{i=1}^n\sum_{j\in\Omega_i}\|\mathbf{x}_i-\mathbf{x}_j\|_2^2.\label{logistic}
\end{align}
The problem \eqref{logistic} is clearly in the form of \eqref{NCM} with:
\begin{align}
&f_i(\mathbf{x}_i)=\sum_{l=1}^q\log\left(1+\exp\left(-t_{il}\mathbf{u}_{il}^\mathsf{T}\mathbf{x}_i\right)\right),\\
&g_{ij}(\mathbf{x}_i,\mathbf{x}_j)=\beta\|\mathbf{x}_i-\mathbf{x}_j\|_2^2.
\end{align}
We note that $f_i$ and $g_{ij}$ are all convex, i.e., they satisfy Assumption 1. In addition, $\nabla f_i$ is Lipschitz continuous with constant $\frac{1}{4}\sum_{l=1}^q\|\mathbf{u}_{il}\|_2^2$ and $\nabla g_{ij}$ is Lipschitz continuous with constant $4\beta$. So, Assumption 2 holds with $L=\max\left\{4\beta,\frac{1}{4}\max_{i=1,...,n}\sum_{l=1}^q\|\mathbf{u}_{il}\|_2^2\right\}$. Thus, Theorem 1 can be applied with appropriate algorithm parameters and convergence of the DLADMM algorithm is guaranteed theoretically. Moreover, though neither $f_i$ nor $g_{ij}$ is strongly convex (Assumption 3 does not hold), we can still empirically observe linear convergence of the DLADMM in later experiments.

\subsection{Comparison between the DLADMM and the DADMM}
We first conduct an experiment to compare the performance of the DLADMM and the DADMM algorithms. We consider two scenarios: (i) a random network with $n=10$ nodes; the dimension of each data sample is $p=2$; and each node has $q=50$ data samples; (ii) a random network with $n=30$ nodes; the dimension of each data instance is $p=5$; and each node has $q=10$ data samples. The average degree of the network is 2. The ADMM algorithm parameter is set to be $\rho=50$ and the linearization parameter is $c=3$ in scenario (i) and $c=5$ in scenario (ii). In Fig. \ref{comparison}, we compare the relative errors $\frac{\left\|\mathbf{x}^k-\mathbf{x}^*\right\|_2}{\left\|\mathbf{x}^*\right\|_2}$ ($\mathbf{x}^*$ is the optimal point of \eqref{logistic} obtained by solving the centralized optimization problem with the CVX package \cite{cvx,gb08}) of the DADMM algorithm and the DLADMM algorithm. We observe that the convergence curve of the DLADMM algorithm is very close to that of the DADMM algorithm in both scenarios. Both the DADMM and the DLADMM converge linearly to the optimal point. However, the computational complexity of the DLADMM is much lower than that of the DADMM. It takes several hours for the DADMM to finish 400 iterations while the DLADMM only needs about 5 seconds to finish the same number of iterations. The reason is that, for each node, each iteration of the DADMM necessitates solving a local optimization problem containing log functions, which must be approximated iteratively. Thus, each iteration of the DADMM is carried out very slowly. On the contrary, each iteration of the DLADMM only involves direct closed-form computations, which can be implemented very quickly. This endows the DLADMM with great computational advantage over the DADMM.

\begin{figure}
  \centering
  \includegraphics[scale=.2]{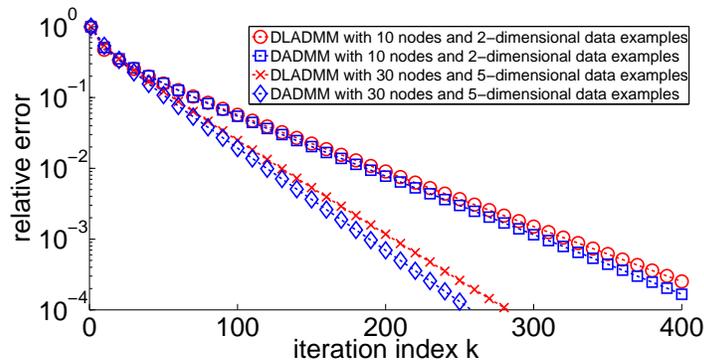}\\
  \caption{Comparison between the DLADMM and DADMM}\label{comparison}
\end{figure}

\subsection{Impact of Network Topology}

Next, we investigate the impact of network topology on the performance of the DLADMM. We set the total number of nodes to be $n=20$ and the algorithm parameters to be $\rho=100,c=50$. We consider four network topologies: the line network, the star network, the complete network and the small-world network. The four network topologies are illustrated in Fig. \ref{topologies}. To obtain small-world network, we first generate a cycle network, and then add 20 random links between them. As its name suggests, in small-world networks, the distance between two nodes, i.e., the length of the shortest path connecting these two nodes, is small. Many properties of real-world networks can be obtained by the small-world networks \cite{watts1998collective}. The convergence curves of the DLADMM on different network topologies are shown in Fig. \ref{topology_comparison}. We observe that the convergence of the small-world network and the star network are faster than that of the line network and complete network. The phenomenon can be explained as follows. For the complete network, the number of constraints in the ADMM formulation of the network cost problem \eqref{admm_formulation}, i.e., the number of nodes plus the number of links, is large. Thus, the number of dual variables at each node is also large, resulting in slow convergence. For the line network, the distance between nodes is generally large, so that information from a node cannot propagate quickly to many distant nodes. This also prohibits the DLADMM from fast convergence. In contrast, for the star network and the small-world network: (i) the distances between nodes are small so that information can be efficiently diffused; (ii) the average degree of nodes is small so that each node only has a small number of dual variables to update, which can converge quickly. Lastly, we remark that though the DLADMM converges at different speeds for different network topologies, it converges linearly to the optimal point in all circumstances.

\begin{figure}
\renewcommand\figurename{\small Fig.}
\centering \vspace*{8pt} \setlength{\baselineskip}{10pt}

\subfigure[Line network]{
\includegraphics[scale = 0.1]{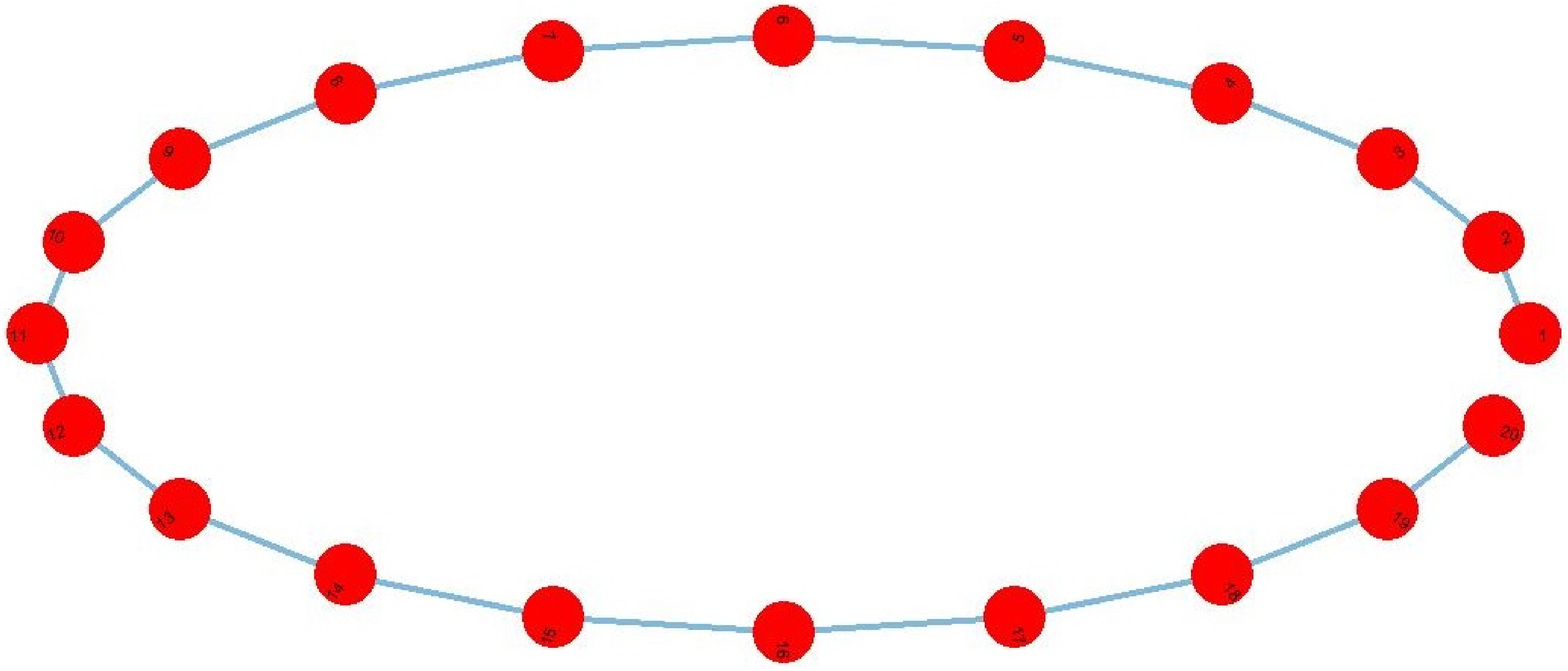}}
\subfigure[Star network]{
\includegraphics[scale = 0.1]{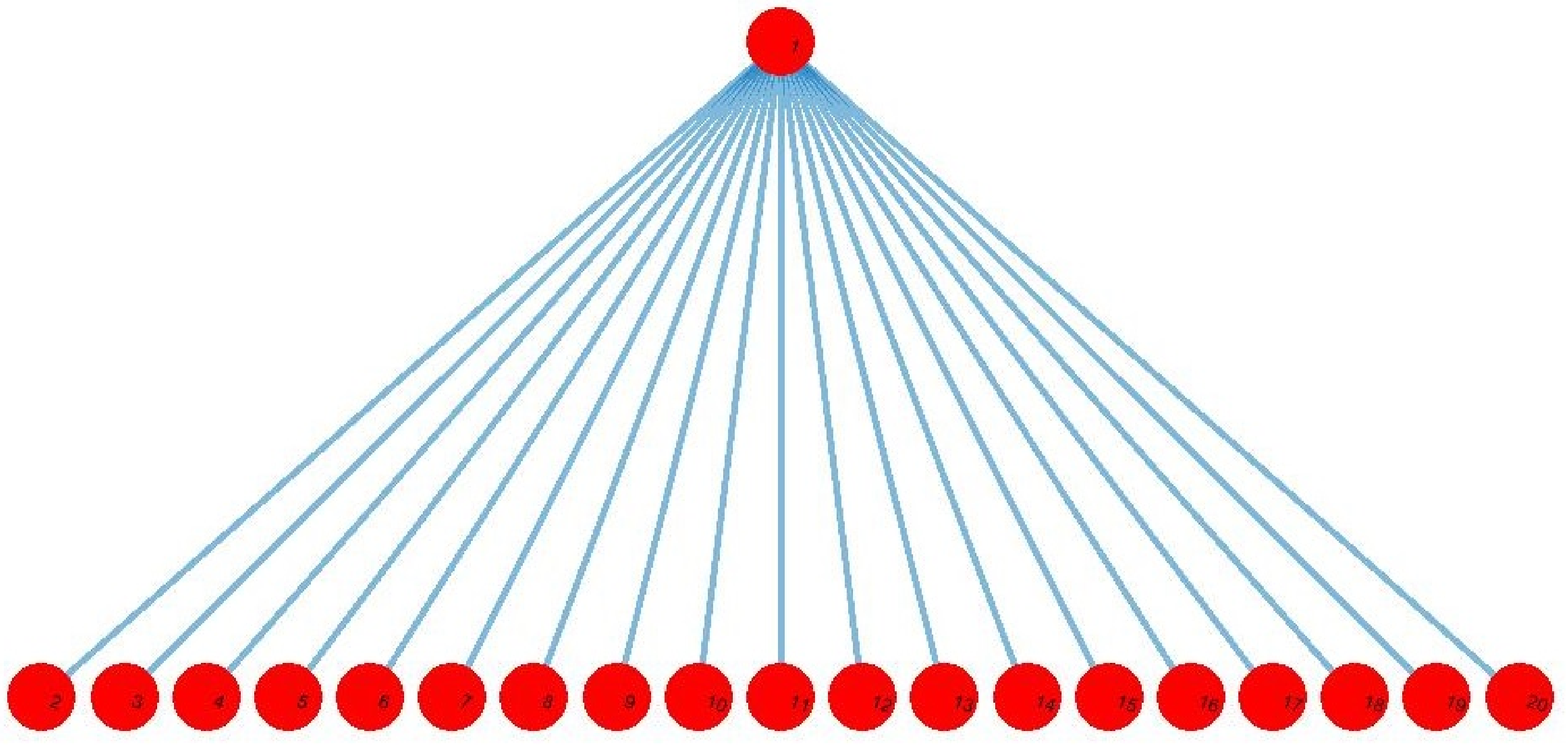}}
\subfigure[Complete network]{
\includegraphics[scale = 0.1]{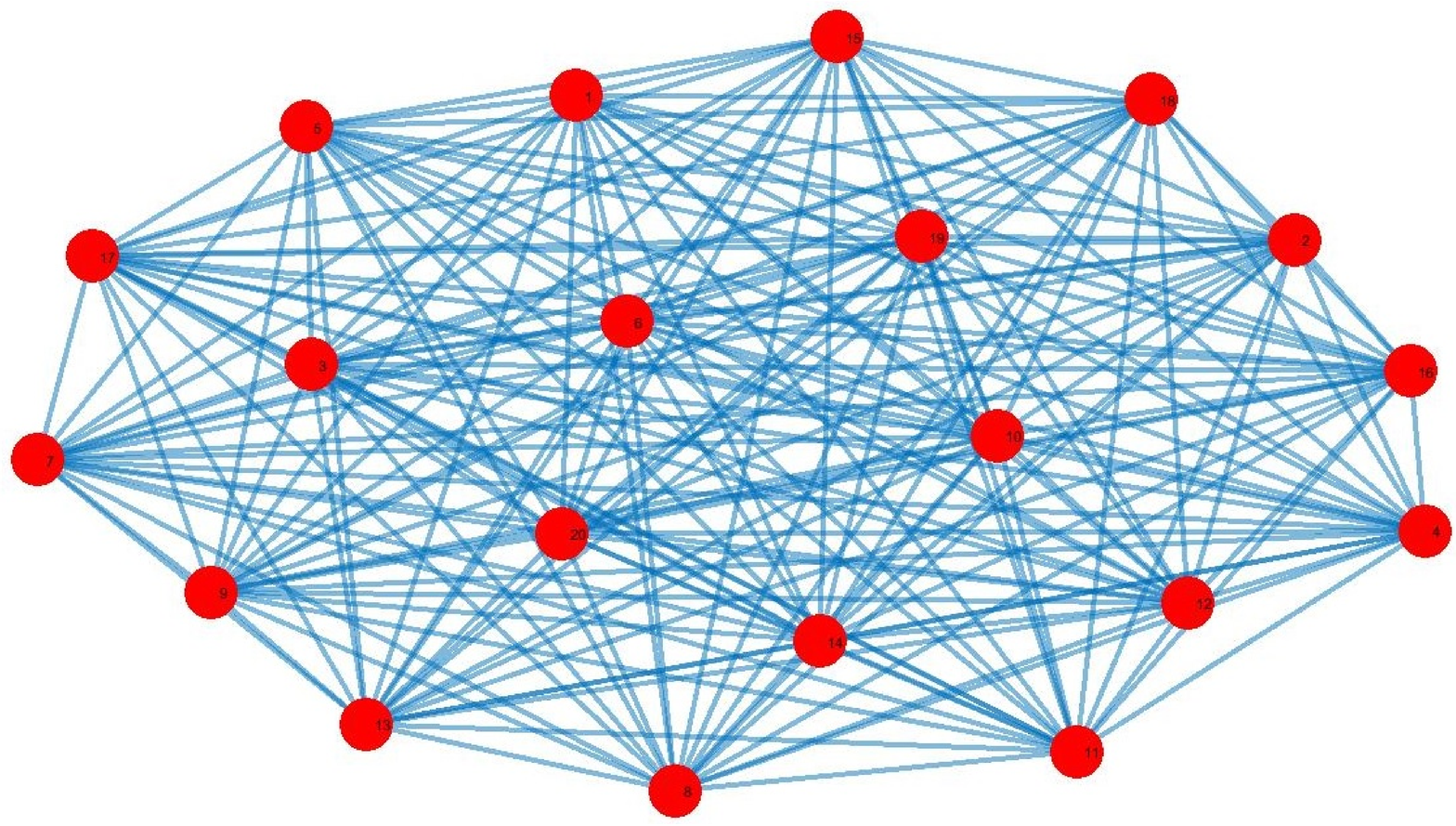}}
\subfigure[Small-world network]{
\includegraphics[scale = 0.1]{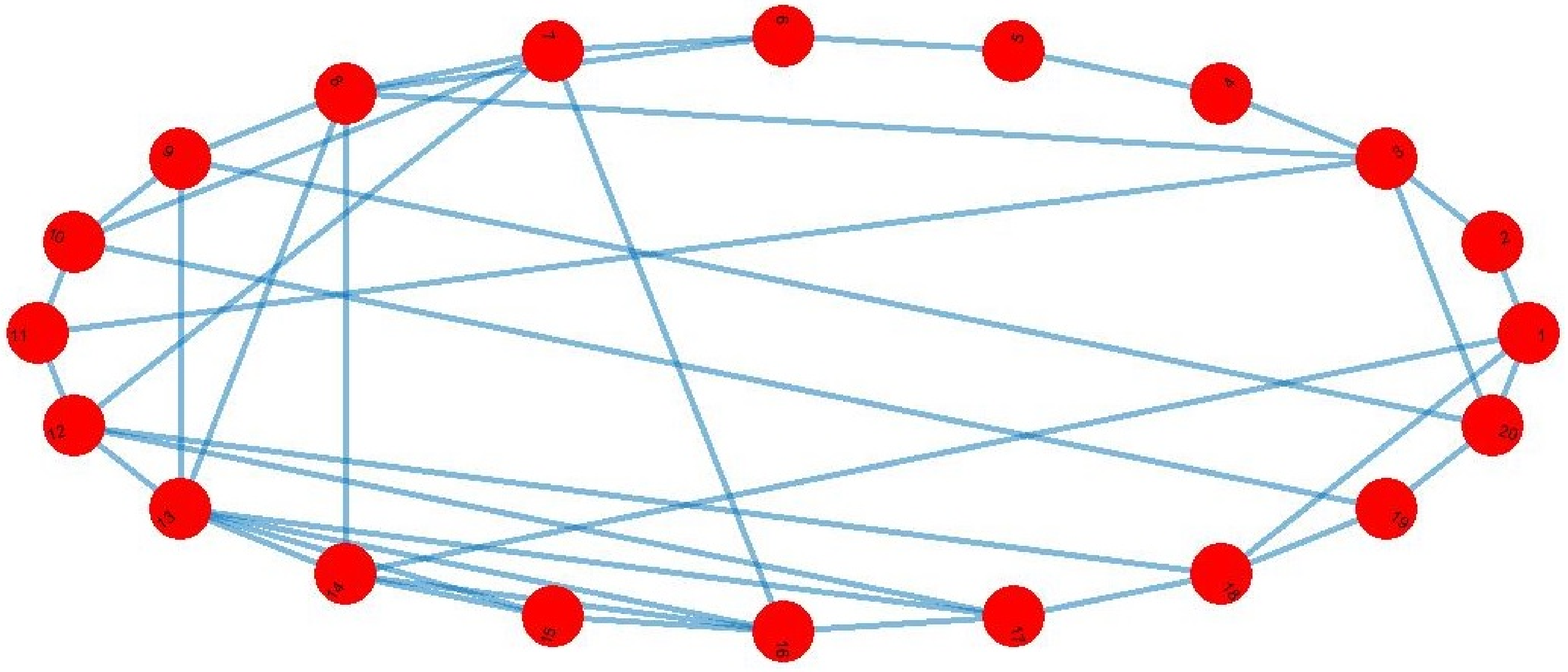}}

\caption{Different network topologies}
\label{topologies}
\end{figure}

\begin{figure}
  \centering
  \includegraphics[scale=.2]{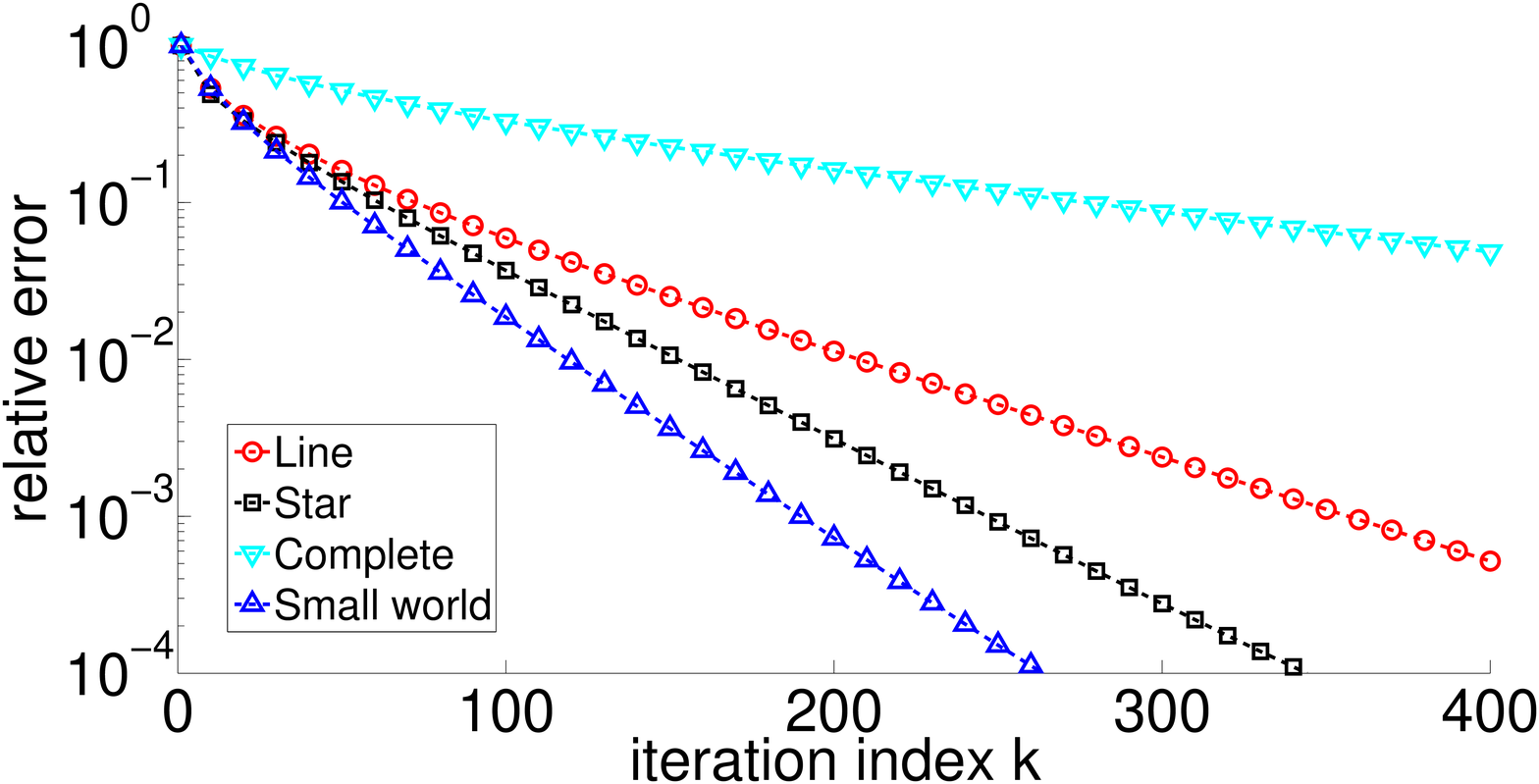}\\
  \caption{Performance of the DLADMM on different network topologies}\label{topology_comparison}
\end{figure}

\subsection{Impact of Network Connectivity}

We further study the impact of network connectivity, measured by the average node degree, on the performance of the DLADMM over small-world networks. To this end, we first form a cycle network and then add different numbers of random links to obtain small-world networks of different average degrees. The convergence curves of the DLADMM algorithm on small-world networks with different average degrees are reported in Fig. \ref{degree}. We observe that the small-world networks with smaller average degree have faster convergence speed. The reason is that for small-world networks, even when the average degree is small (e.g., 3), the distances between nodes are short so that information of one node can spread across the network quickly. Additionally, for small-world network with lower degrees, each node only needs to update a small number of dual variables and thus the convergence is faster. Note that when the average degree is high, the small-world networks become analogous to the complete network, over which the DLADMM converges slowly (Fig. \ref{topology_comparison}).

\begin{figure}
  \centering
  \includegraphics[scale=.2]{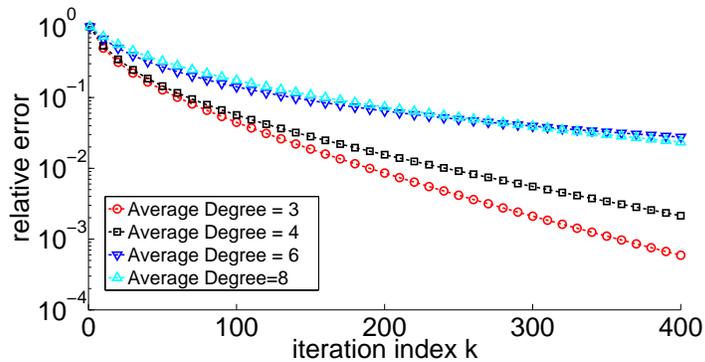}\\
  \caption{Performance of the DLADMM on the small-world networks with different average degrees}\label{degree}
\end{figure}

\subsection{Impact of the Linearization Parameter $c$}

Finally,  in Fig. \ref{impact_of_c}, we study the impact of the linearization parameter $c$ on the convergence of the DLADMM over small-world networks. We observe that as long as the DLADMM converges, the smaller the value of $c$ , the faster the convergence speed. But $c$ cannot be too small, otherwise the DLADMM may diverge, e.g., when $c=1$. Recall that the parameter $c$ is introduced to limit the step size between consecutive iterations and therefore plays a similar role as the step size parameter in numerical optimization \cite{boyd2004convex} and adaptive signal processing \cite{Haykin:1996:AFT:230061}. A general tradeoff for such parameters is that (i) when they are too large, the convergence is slow; (ii) when they are too small, the algorithm risks divergence. We note that similar phenomenon can be observed in Fig. \ref{impact_of_c}.

\begin{figure}
  \centering
  \includegraphics[scale=.2]{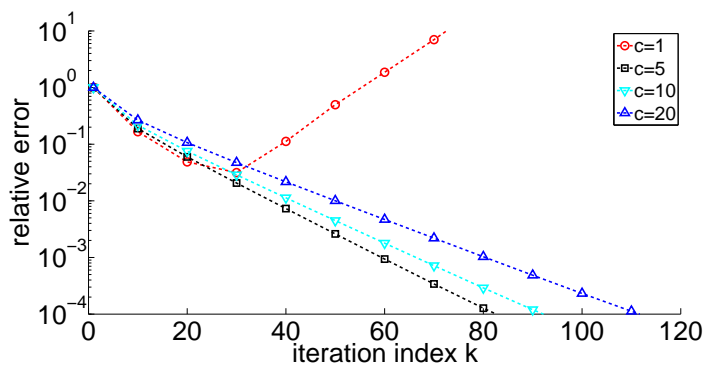}\\
  \caption{Performance of the DLADMM with different values of $c$}\label{impact_of_c}
\end{figure}

\section{Conclusions}
In this paper, we study the generic form of network cost minimization problem, in which the network cost includes both node costs and link costs. The formulated problem has broad applications in distributed signal processing and control over multi-agent networked systems. A distributed linearized ADMM algorithm is presented for the formulated problem. The DLADMM algorithm operates in a decentralized manner and each iteration only involves simple closed-form computations, which endows the DLADMM much lower computational complexity than the distributed ADMM. Under the assumptions that the local cost functions are convex and possess Lipschitz continuous gradients, we show that the DLADMM converges to an optimal point of the network cost minimization problem. By further assuming that the local cost functions are strongly convex, we can guarantee linear convergence rate of the DLADMM. Numerical simulations are carried out to validate the performance of the DLADMM and we empirically observe that the DLADMM has similar convergence performance as DADMM does while the former has much lower computational overhead. The impacts of network topology, connectivity and algorithm parameters on the convergence behaviors of the DLADMM are also discussed.

\section*{Appendix A\\Proof of Lemma \ref{lemA}}
For any $\mathbf{x},\mathbf{x}'\in\mathbb{R}^{np}$:
\begin{align}
\left\|\nabla f(\mathbf{x})-\nabla f\left(\mathbf{x}'\right)\right\|_2^2=\sum_{i=1}^n\left\|\nabla f_i(\mathbf{x}_i)-\nabla f_i\left(\mathbf{x}_i'\right)\right\|_2^2\leq L^2\sum_{i=1}^n\left\|\mathbf{x}_i-\mathbf{x}_i'\right\|_2^2=L^2\left\|\mathbf{x}-\mathbf{x}'\right\|_2^2.
\end{align}
So, $\nabla f$ is Lipschitz continuous with constant $L$. For any $\mathbf{y,y'}\in\mathbb{R}^{np}$ and $\mathbf{z,z'}\in\mathbb{R}^{mp}$:
\begin{align}
&\left\|\nabla g(\mathbf{y,z})-\nabla g(\mathbf{y',z'})\right\|_2^2\\
\label{g}&=\sum_{i=1}^n\left\|\sum_{j\in\Omega_i}\left(\nabla_{\mathbf{y}_i}g_{ij}(\mathbf{y}_i,\mathbf{z}_{ij})-\nabla_{\mathbf{y}_i'}g_{ij}\left(\mathbf{y}_i',\mathbf{z}_{ij}'\right)\right)\right\|_2^2+\sum_{i=1}^n\sum_{j\in\Omega_i}\left\|\nabla_{\mathbf{z}_{ij}}g_{ij}(\mathbf{y}_i,\mathbf{z}_{ij})-\nabla_{\mathbf{z}_{ij}'}g_{ij}\left(\mathbf{y}_i',\mathbf{z}_{ij}'\right)\right\|_2^2
\end{align}
For the first term, we have:
\begin{align}
&\left\|\sum_{j\in\Omega_i}\left(\nabla_{\mathbf{y}_i}g_{ij}(\mathbf{y}_i,\mathbf{z}_{ij})-\nabla_{\mathbf{y}_i'}g_{ij}\left(\mathbf{y}_i',\mathbf{z}_{ij}'\right)\right)\right\|_2^2\\
&\leq\left(\sum_{j\in\Omega_i}\left\|\nabla_{\mathbf{y}_i}g_{ij}(\mathbf{y}_i,\mathbf{z}_{ij})-\nabla_{\mathbf{y}_i'}g_{ij}\left(\mathbf{y}_i',\mathbf{z}_{ij}'\right)\right\|_2\right)^2\\
&\leq L^2\left(\sum_{j\in\Omega_i}
\left\|
\left[
\begin{array}{c}
\mathbf{y}_i\\
\mathbf{z}_{ij}
\end{array}
\right]-
\left[
\begin{array}{c}
\mathbf{y}_i'\\
\mathbf{z}_{ij}'
\end{array}
\right]
\right\|_2
\right)^2\\
\label{cauchy}&\leq L^2|\Omega_i|\sum_{j\in\Omega_i}\left\|
\left[
\begin{array}{c}
\mathbf{y}_i\\
\mathbf{z}_{ij}
\end{array}
\right]-
\left[
\begin{array}{c}
\mathbf{y}_i'\\
\mathbf{z}_{ij}'
\end{array}
\right]
\right\|_2^2\\
&\leq L^2K\sum_{j\in\Omega_i}\left\|
\left[
\begin{array}{c}
\mathbf{y}_i\\
\mathbf{z}_{ij}
\end{array}
\right]-
\left[
\begin{array}{c}
\mathbf{y}_i'\\
\mathbf{z}_{ij}'
\end{array}
\right]
\right\|_2^2,
\end{align}
where we invoke Cauchy's inequality $\left(\sum_{i=1}^Iv_i\right)^2\leq I\sum_{i=1}^Iv_i^2,\forall\mathbf{v}\in\mathbb{R}^I$ to obtain \eqref{cauchy}. Hence,
\begin{align}
\sum_{i=1}^n\left\|\sum_{j\in\Omega_i}\left(\nabla_{\mathbf{y}_i}g_{ij}(\mathbf{y}_i,\mathbf{z}_{ij})-\nabla_{\mathbf{y}_i'}g_{ij}\left(\mathbf{y}_i',\mathbf{z}_{ij}'\right)\right)\right\|_2^2
\leq L^2K^2\left\|
\left[
\begin{array}{c}
\mathbf{y}\\
\mathbf{z}
\end{array}
\right]-
\left[
\begin{array}{c}
\mathbf{y}'\\
\mathbf{z}'
\end{array}
\right]
\right\|_2^2\label{g1}
\end{align}
For the second term of \eqref{g}, we have:
\begin{align}
&\sum_{i=1}^n\sum_{j\in\Omega_i}\left\|\nabla_{\mathbf{z}_{ij}}g_{ij}(\mathbf{y}_i,\mathbf{z}_{ij})-\nabla_{\mathbf{z}_{ij}'}g_{ij}\left(\mathbf{y}_i',\mathbf{z}_{ij}'\right)\right\|_2^2\\
&\leq L^2\sum_{i=1}^n\sum_{j\in\Omega_i}\left\|
\left[
\begin{array}{c}
\mathbf{y}_i\\
\mathbf{z}_{ij}
\end{array}
\right]-
\left[
\begin{array}{c}
\mathbf{y}_i'\\
\mathbf{z}_{ij}'
\end{array}
\right]
\right\|_2^2\\
&\leq L^2K\left\|
\left[
\begin{array}{c}
\mathbf{y}\\
\mathbf{z}
\end{array}
\right]-
\left[
\begin{array}{c}
\mathbf{y}'\\
\mathbf{z}'
\end{array}
\right]
\right\|_2^2\label{g2}
\end{align}
Combining \eqref{g1} and \eqref{g2} yields:
\begin{align}
\left\|\nabla g(\mathbf{y,z})-\nabla g(\mathbf{y',z'})\right\|_2^2\leq\left(L^2K^2+L^2K\right)\left\|
\left[
\begin{array}{c}
\mathbf{y}\\
\mathbf{z}
\end{array}
\right]-
\left[
\begin{array}{c}
\mathbf{y}'\\
\mathbf{z}'
\end{array}
\right]
\right\|_2^2.
\end{align}
So $\nabla g$ is Lipschitz continuous with constant $M=\sqrt{L^2K^2+L^2K}$.

\section*{Appendix B\\Proof of Proposition 1}
According to Assumption 1, problem \eqref{admm_formulation_2} is a convex optimization problem. Thus, the Karush-Kuhn-Tucker (KKT) conditions are necessary and sufficient for optimality. So, the primal/dual optimal point $\mathbf{u}^*$ satisfies the following KKT conditions:
\begin{align}
\label{x_kkt}&\nabla f(\mathbf{x}^*)+\mathbf{B}^\mathsf{T}\boldsymbol{\alpha}^*=\mathbf{0},\\
\label{w_kkt}&\nabla g(\mathbf{w}^*)-\boldsymbol{\alpha}^*=\mathbf{0},\\
\label{feasible}&\mathbf{Bx}^*-\mathbf{w}^*=\mathbf{0}.
\end{align}
Subtracting \eqref{x_kkt} from \eqref{x_dladmm_c} and exploiting \eqref{feasible} yields gives:
\begin{align}\label{ff}
\nabla f\left(\mathbf{x}^k\right)-\nabla f(\mathbf{x}^*)+c\left(\mathbf{x}^{k+1}-\mathbf{x}^k\right)+\mathbf{B}^\mathsf{T}\left(\boldsymbol{\alpha}^k-\boldsymbol{\alpha}^*\right)+\rho\mathbf{B}^\mathsf{T}\mathbf{B}\left(\mathbf{x}^{k+1}-\mathbf{x}^*\right)+\rho\mathbf{B}^\mathsf{T}\left(\mathbf{w}^*-\mathbf{w}^k\right)=\mathbf{0}.
\end{align}
Similarly, subtracting \eqref{w_kkt} from \eqref{w_dladmm_c} and exploiting \eqref{feasible} yields:
\begin{align}\label{gg}
\nabla g\left(\mathbf{w}^k\right)-\nabla g\left(\mathbf{w}^*\right)+c\left(\mathbf{w}^{k+1}-\mathbf{w}^k\right)+\boldsymbol{\alpha}^*-\boldsymbol{\alpha}^k+\rho\left(\mathbf{w}^{k+1}-\mathbf{w}^*\right)+\rho\mathbf{B}\left(\mathbf{x}^*-\mathbf{x}^{k+1}\right)=\mathbf{0}.
\end{align}
Combining \eqref{alpha_dladmm_c} and \eqref{feasible} leads to:
\begin{align}\label{aa}
\boldsymbol{\alpha}^{k+1}-\boldsymbol{\alpha}^k+\rho\mathbf{B}\left(\mathbf{x}^*-\mathbf{x}^{k+1}\right)+\rho\left(\mathbf{w}^{k+1}-\mathbf{w}^*\right)=\mathbf{0}.
\end{align}
According to Lemma \ref{lemE}, we have:
\begin{align}
&\frac{1}{L}\left\|\nabla f\left(\mathbf{x}^k\right)-\nabla f(\mathbf{x}^*)\right\|_2^2\\
&\leq\left(\mathbf{x}^k-\mathbf{x}^*\right)^\mathsf{T}\left(\nabla f\left(\mathbf{x}^k\right)-\nabla f(\mathbf{x}^*)\right)\\
&\label{1}=\left(\mathbf{x}^{k+1}-\mathbf{x}^*\right)^\mathsf{T}\left(\nabla f\left(\mathbf{x}^k\right)-\nabla f(\mathbf{x}^*)\right)+\left(\mathbf{x}^k-\mathbf{x}^{k+1}\right)^\mathsf{T}\left(\nabla f\left(\mathbf{x}^k\right)-\nabla f(\mathbf{x}^*)\right).
\end{align}
For the first term of \eqref{1}, according to \eqref{ff}, we have:
\begin{align}
&\left(\mathbf{x}^{k+1}-\mathbf{x}^*\right)^\mathsf{T}\left(\nabla f\left(\mathbf{x}^k\right)-\nabla f(\mathbf{x}^*)\right)\\
&=\left(\mathbf{x}^{k+1}-\mathbf{x}^*\right)^\mathsf{T}\bigg[c\left(\mathbf{x}^k-\mathbf{x}^{k+1}\right)+\mathbf{B}^\mathsf{T}\left(\boldsymbol{\alpha}^*-\boldsymbol{\alpha}^k\right)+\rho\mathbf{B}^\mathsf{T}\mathbf{B}\left(\mathbf{x}^*-\mathbf{x}^{k+1}\right)+\rho\mathbf{B}^\mathsf{T}\left(\mathbf{w}^{k+1}-\mathbf{w}^*\right)\nonumber\\
&~~~+\rho\mathbf{B}^\mathsf{T}\left(\mathbf{w}^k-\mathbf{w}^{k+1}\right)\bigg]\label{cc1}.
\end{align}
On the other hand, using Lemma \ref{lemE} again, we have:
\begin{align}
&\frac{1}{M}\left\|\nabla g\left(\mathbf{w}^k\right)-\nabla g(\mathbf{w}^*)\right\|_2^2\\
&\leq\left(\mathbf{w}^k-\mathbf{w}^*\right)^\mathsf{T}\left(\nabla g\left(\mathbf{w}^k\right)-\nabla g(\mathbf{w}^*)\right)\\
&\label{2}=\left(\mathbf{w}^{k+1}-\mathbf{w}^*\right)^\mathsf{T}\left(\nabla g\left(\mathbf{w}^k\right)-\nabla g(\mathbf{w}^*)\right)+\left(\mathbf{w}^k-\mathbf{w}^{k+1}\right)^\mathsf{T}\left(\nabla g\left(\mathbf{w}^k\right)-\nabla g(\mathbf{w}^*)\right).
\end{align}
For the first term of \eqref{2}, by using \eqref{gg}, we obtain:
\begin{align}
&\left(\mathbf{w}^{k+1}-\mathbf{w}^*\right)^\mathsf{T}\left(\nabla g\left(\mathbf{w}^k\right)-\nabla g(\mathbf{w}^*)\right)\\
\label{cc2}&=\left(\mathbf{w}^{k+1}-\mathbf{w}^*\right)^\mathsf{T}\left[c\left(\mathbf{w}^k-\mathbf{w}^{k+1}\right)+\boldsymbol{\alpha}^k-\boldsymbol{\alpha}^*+\rho\left(\mathbf{w}^*-\mathbf{w}^{k+1}\right)+\rho\mathbf{B}\left(\mathbf{x}^{k+1}-\mathbf{x}^*\right)\right]
\end{align}
We note the following fact, which shall be used frequently.
\begin{lem}\label{lemF}
For any symmetric matrix $\mathbf{A}\in\mathbb{R}^{N\times N}$ and any vectors $\mathbf{x,y,z}\in\mathbb{R}^N$, we have:
\begin{align}
2(\mathbf{x-y})^\mathsf{T}\mathbf{A(z-x)}=\mathbf{(z-y)}^\mathsf{T}\mathbf{A(z-y)}-\mathbf{(x-y)}^\mathsf{T}\mathbf{A(x-y)}-\mathbf{(z-x)}^\mathsf{T}\mathbf{A(z-x)}.
\end{align}
\end{lem}
Making use of Lemma \ref{lemF}, we obtain:
\begin{align}\label{c1}
c\left(\mathbf{x}^{k+1}-\mathbf{x}^*\right)^\mathsf{T}\left(\mathbf{x}^k-\mathbf{x}^{k+1}\right)=\frac{c}{2}\left\|\mathbf{x}^k-\mathbf{x}^*\right\|_2^2-\frac{c}{2}\left\|\mathbf{x}^{k+1}-\mathbf{x}^*\right\|_2^2-\frac{c}{2}\left\|\mathbf{x}^k-\mathbf{x}^{k+1}\right\|_2^2,
\end{align}
and
\begin{align}\label{c2}
c\left(\mathbf{w}^{k+1}-\mathbf{w}^*\right)^\mathsf{T}\left(\mathbf{w}^k-\mathbf{w}^{k+1}\right)=\frac{c}{2}\left\|\mathbf{w}^k-\mathbf{w}^*\right\|_2^2-\frac{c}{2}\left\|\mathbf{w}^{k+1}-\mathbf{w}^*\right\|_2^2-\frac{c}{2}\left\|\mathbf{w}^k-\mathbf{w}^{k+1}\right\|_2^2.
\end{align}
Based on \eqref{aa}, we get:
\begin{align}
&\left(\mathbf{x}^{k+1}-\mathbf{x}^*\right)^\mathsf{T}\mathbf{B}^\mathsf{T}\left(\boldsymbol{\alpha}^*-\boldsymbol{\alpha}^k\right)+\left(\mathbf{w}^{k+1}-\mathbf{w}^*\right)^\mathsf{T}\left(\boldsymbol{\alpha}^k-\boldsymbol{\alpha}^*\right)\\
&=\left[-\mathbf{B}\left(\mathbf{x}^{k+1}-\mathbf{x}^*\right)+\mathbf{w}^{k+1}-\mathbf{w}^*\right]^\mathsf{T}\left(\boldsymbol{\alpha}^k-\boldsymbol{\alpha}^*\right)\\
&=\frac{1}{\rho}\left(\boldsymbol{\alpha}^k-\boldsymbol{\alpha}^{k+1}\right)^\mathsf{T}\left(\boldsymbol{\alpha}^k-\boldsymbol{\alpha}^*\right)\\
\label{c3}&=-\frac{1}{2\rho}\left\|\boldsymbol{\alpha}^*-\boldsymbol{\alpha}^{k+1}\right\|_2^2+\frac{1}{2\rho}\left\|\boldsymbol{\alpha}^k-\boldsymbol{\alpha}^{k+1}\right\|_2^2+\frac{1}{2\rho}\left\|\boldsymbol{\alpha}^*-\boldsymbol{\alpha}^k\right\|_2^2.
\end{align}
Again, using \eqref{aa}, we obtain:
\begin{align}
&\left(\mathbf{x}^{k+1}-\mathbf{x}^*\right)^\mathsf{T}\left[\rho\mathbf{B}^\mathsf{T}\mathbf{B}\left(\mathbf{x}^*-\mathbf{x}^{k+1}\right)+\rho\mathbf{B}^\mathsf{T}\left(\mathbf{w}^{k+1}-\mathbf{w}^*\right)\right]\nonumber\\
&+\left(\mathbf{w}^{k+1}-\mathbf{w}^*\right)^\mathsf{T}\left[\rho\left(\mathbf{w}^*-\mathbf{w}^{k+1}\right)+\rho\mathbf{B}\left(\mathbf{x}^{k+1}-\mathbf{x}^*\right)\right]\\
&=-\rho\left\|\mathbf{B}\left(\mathbf{x}^{k+1}-\mathbf{x}^*\right)+\mathbf{w}^*-\mathbf{w}^{k+1}\right\|_2^2\\
&=-\frac{1}{\rho}\left\|\boldsymbol{\alpha}^{k+1}-\boldsymbol{\alpha}^k\right\|_2^2.\label{c4}
\end{align}
Once again, using \eqref{aa}, we have:
\begin{align}
&\rho\left(\mathbf{x}^{k+1}-\mathbf{x}^*\right)^\mathsf{T}\mathbf{B}^\mathsf{T}\left(\mathbf{w}^k-\mathbf{w}^{k+1}\right)\\
&=\left[\boldsymbol{\alpha}^{k+1}-\boldsymbol{\alpha}^k+\rho\left(\mathbf{w}^{k+1}-\mathbf{w}^*\right)\right]^\mathsf{T}\left(\mathbf{w}^k-\mathbf{w}^{k+1}\right)\\
&=\left(\boldsymbol{\alpha}^{k+1}-\boldsymbol{\alpha}^k\right)^\mathsf{T}\left(\mathbf{w}^k-\mathbf{w}^{k+1}\right)+\frac{\rho}{2}\left\|\mathbf{w}^k-\mathbf{w}^*\right\|_2^2-\frac{\rho}{2}\left\|\mathbf{w}^{k+1}-\mathbf{w}^*\right\|_2^2-\frac{\rho}{2}\left\|\mathbf{w}^k-\mathbf{w}^{k+1}\right\|_2^2\\
&\leq\frac{1}{2}\left\|\frac{1}{\sqrt{2\rho}}\left(\boldsymbol{\alpha}^{k+1}-\boldsymbol{\alpha}^k\right)\right\|_2^2+\frac{1}{2}\left\|\sqrt{2\rho}\left(\mathbf{w}^k-\mathbf{w}^{k+1}\right)\right\|_2^2+\frac{\rho}{2}\left\|\mathbf{w}^k-\mathbf{w}^*\right\|_2^2-\frac{\rho}{2}\left\|\mathbf{w}^{k+1}-\mathbf{w}^*\right\|_2^2\nonumber\\
&~~-\frac{\rho}{2}\left\|\mathbf{w}^k-\mathbf{w}^{k+1}\right\|_2^2\\
\label{c5}&=\frac{1}{4\rho}\left\|\boldsymbol{\alpha}^{k+1}-\boldsymbol{\alpha}^k\right\|_2^2+\frac{\rho}{2}\left\|\mathbf{w}^k-\mathbf{w}^{k+1}\right\|_2^2+\frac{\rho}{2}\left\|\mathbf{w}^k-\mathbf{w}^*\right\|_2^2-\frac{\rho}{2}\left\|\mathbf{w}^{k+1}-\mathbf{w}^*\right\|_2^2.
\end{align}
Adding the results in \eqref{c1}, \eqref{c2}, \eqref{c3}, \eqref{c4} and \eqref{c5} all together and noting \eqref{cc1} and \eqref{cc2}, we get:
\begin{align}
&\left(\mathbf{x}^{k+1}-\mathbf{x}^*\right)^\mathsf{T}\left(\nabla f\left(\mathbf{x}^k\right)-\nabla f(\mathbf{x}^*)\right)+\left(\mathbf{w}^{k+1}-\mathbf{w}^*\right)^\mathsf{T}\left(\nabla g\left(\mathbf{w}^k\right)-\nabla g(\mathbf{w}^*)\right)\\
&\leq\frac{c}{2}\left\|\mathbf{x}^k-\mathbf{x}^*\right\|_2^2-\frac{c}{2}\left\|\mathbf{x}^{k+1}-\mathbf{x}^*\right\|_2^2-\frac{c}{2}\left\|\mathbf{x}^k-\mathbf{x}^{k+1}\right\|_2^2+\frac{c}{2}\left\|\mathbf{w}^k-\mathbf{w}^*\right\|_2^2-\frac{c}{2}\left\|\mathbf{w}^{k+1}-\mathbf{w}^*\right\|\nonumber\\
&~~-\frac{c}{2}\left\|\mathbf{w}^k-\mathbf{w}^{k+1}\right\|_2^2-\frac{1}{2\rho}\left\|\boldsymbol{\alpha}^*-\boldsymbol{\alpha}^{k+1}\right\|_2^2+\frac{1}{2\rho}\left\|\boldsymbol{\alpha}^*-\boldsymbol{\alpha}^k\right\|_2^2-\frac{1}{4\rho}\left\|\boldsymbol{\alpha}^{k+1}-\boldsymbol{\alpha}^k\right\|_2^2+\frac{\rho}{2}\left\|\mathbf{w}^k-\mathbf{w}^*\right\|_2^2\nonumber\\
&~~-\frac{\rho}{2}\left\|\mathbf{w}^{k+1}-\mathbf{w}^*\right\|_2^2+\frac{\rho}{2}\left\|\mathbf{w}^k-\mathbf{w}^{k+1}\right\|_2^2\\
&=\left\|\mathbf{u}^k-\mathbf{u}^*\right\|_\mathbf{\Lambda}^2-\left\|\mathbf{u}^{k+1}-\mathbf{u}^*\right\|_\mathbf{\Lambda}^2-\frac{c}{2}\left\|\mathbf{x}^k-\mathbf{x}^{k+1}\right\|_2^2-\frac{c-\rho}{2}\left\|\mathbf{w}^k-\mathbf{w}^{k+1}\right\|_2^2-\frac{1}{4\rho}\left\|\boldsymbol{\alpha}^{k+1}-\boldsymbol{\alpha}^k\right\|_2^2.\label{3}
\end{align}
Combining \eqref{1}, \eqref{2} and \eqref{3}, we obtain:
\begin{align}
&\frac{1}{L}\left\|\nabla f\left(\mathbf{x}^k\right)-\nabla f(\mathbf{x}^*)\right\|_2^2+\frac{1}{M}\left\|\nabla g\left(\mathbf{w}^k\right)-\nabla g(\mathbf{w}^*)\right\|_2^2\\
&\leq\left\|\mathbf{u}^k-\mathbf{u}^*\right\|_\mathbf{\Lambda}^2-\left\|\mathbf{u}^{k+1}-\mathbf{u}^*\right\|_\mathbf{\Lambda}^2-\frac{c}{2}\left\|\mathbf{x}^k-\mathbf{x}^{k+1}\right\|_2^2-\frac{c-\rho}{2}\left\|\mathbf{w}^k-\mathbf{w}^{k+1}\right\|_2^2-\frac{1}{4\rho}\left\|\boldsymbol{\alpha}^{k+1}-\boldsymbol{\alpha}^k\right\|_2^2\nonumber\\
&~~+\left(\mathbf{x}^k-\mathbf{x}^{k+1}\right)^\mathsf{T}\left(\nabla f\left(\mathbf{x}^k\right)-\nabla f(\mathbf{x}^*)\right)+\left(\mathbf{w}^k-\mathbf{w}^{k+1}\right)^\mathsf{T}\left(\nabla g\left(\mathbf{w}^k\right)-\nabla g(\mathbf{w}^*)\right)\\
&\leq\left\|\mathbf{u}^k-\mathbf{u}^*\right\|_\mathbf{\Lambda}^2-\left\|\mathbf{u}^{k+1}-\mathbf{u}^*\right\|_\mathbf{\Lambda}^2-\frac{c}{2}\left\|\mathbf{x}^k-\mathbf{x}^{k+1}\right\|_2^2-\frac{c-\rho}{2}\left\|\mathbf{w}^k-\mathbf{w}^{k+1}\right\|_2^2-\frac{1}{4\rho}\left\|\boldsymbol{\alpha}^{k+1}-\boldsymbol{\alpha}^k\right\|_2^2\nonumber\\
&~~+\frac{L}{4}\left\|\mathbf{x}^k-\mathbf{x}^{k+1}\right\|_2^2+\frac{1}{L}\left\|\nabla f\left(\mathbf{x}^k\right)-\nabla f(\mathbf{x}^*)\right\|_2^2+\frac{M}{4}\left\|\mathbf{w}^k-\mathbf{w}^{k+1}\right\|_2^2\nonumber\\
&~~+\frac{1}{M}\left\|\nabla g\left(\mathbf{w}^k\right)-\nabla g(\mathbf{w}^*)\right\|_2^2\label{4}
\end{align}
We can rewrite \eqref{4} as:
\begin{align}
\left\|\mathbf{u}^{k+1}-\mathbf{u}^*\right\|_\mathbf{\Lambda}^2\leq&\left\|\mathbf{u}^k-\mathbf{u}^*\right\|_\mathbf{\Lambda}^2-\left(\frac{c}{2}-\frac{L}{4}\right)\left\|\mathbf{x}^k-\mathbf{x}^{k+1}\right\|_2^2-\left(\frac{c-\rho}{2}-\frac{M}{4}\right)\left\|\mathbf{w}^k-\mathbf{w}^{k+1}\right\|_2^2\nonumber\\
&-\frac{1}{4\rho}\left\|\boldsymbol{\alpha}^{k+1}-\boldsymbol{\alpha}^k\right\|_2^2
\end{align}

\section*{Appendix C: Proof of Lemma \ref{lemC}}
Repeating the previous derivation, we know that $\underline{\mathbf{u}}$ and $\overline{\mathbf{u}}$ are both primal/dual optimal points of problem \eqref{admm_formulation_2}. Hence, according to Proposition 1 and $c>\frac{M}{2}+\rho$, we know that $\left\|\mathbf{u}^k-\underline{\mathbf{u}}\right\|_\mathbf{\Lambda}^2$ and $\left\|\mathbf{u}^k-\overline{\mathbf{u}}\right\|_\mathbf{\Lambda}^2$ are decreasing and convergent sequences. Define their limits to be $\underline{\eta}$ and $\overline{\eta}$, respectively, i.e., $\lim_{k\rightarrow\infty}\left\|\mathbf{u}^k-\underline{\mathbf{u}}\right\|_\mathbf{\Lambda}^2=\underline{\eta}$ and $\lim_{k\rightarrow\infty}\left\|\mathbf{u}^k-\overline{\mathbf{u}}\right\|_\mathbf{\Lambda}^2=\overline{\eta}$.
We have:
\begin{align}
&\left\|\mathbf{u}^k-\underline{\mathbf{u}}\right\|_\mathbf{\Lambda}^2-\left\|\mathbf{u}^k-\overline{\mathbf{u}}\right\|_\mathbf{\Lambda}^2\\
&=\frac{c}{2}\left\|\mathbf{x}^k-\underline{\mathbf{x}}\right\|_2^2+\frac{c+\rho}{2}\left\|\mathbf{w}^k-\underline{\mathbf{w}}\right\|_2^2+\frac{1}{2\rho}\left\|\boldsymbol{\alpha}^k-\underline{\boldsymbol{\alpha}}\right\|_2^2\nonumber\\
&~~~-\frac{c}{2}\left\|\mathbf{x}^k-\overline{\mathbf{x}}\right\|_2^2-\frac{c+\rho}{2}\left\|\mathbf{w}^k-\overline{\mathbf{w}}\right\|_2^2-\frac{1}{2\rho}\left\|\boldsymbol{\alpha}^k-\overline{\boldsymbol{\alpha}}\right\|_2^2\\
&\label{8}=c\mathbf{x}^{k\mathsf{T}}(\overline{\mathbf{x}}-\underline{\mathbf{x}})+(c+\rho)\mathbf{w}^{k\mathsf{T}}(\overline{\mathbf{w}}-\underline{\mathbf{w}})+\frac{1}{\rho}\boldsymbol{\alpha}^{k\mathsf{T}}(\overline{\boldsymbol{\alpha}}-\underline{\boldsymbol{\alpha}})\nonumber\\
&~~~+\frac{c}{2}\|\underline{\mathbf{x}}\|_2^2+\frac{c+\rho}{2}\|\underline{\mathbf{w}}\|_2^2+\frac{1}{2\rho}\|\underline{\boldsymbol{\alpha}}\|_2^2-\frac{c}{2}\|\overline{\mathbf{x}}\|_2^2-\frac{c+\rho}{2}\|\overline{\mathbf{w}}\|_2^2-\frac{1}{2\rho}\|\overline{\boldsymbol{\alpha}}\|_2^2\\
&\rightarrow \underline{\eta}-\overline{\eta},~\text{as}~k\rightarrow\infty.
\end{align}
Hence, any subsequence of $\left\|\mathbf{u}^k-\underline{\mathbf{u}}\right\|_\mathbf{\Lambda}^2-\left\|\mathbf{u}^k-\overline{\mathbf{u}}\right\|_\mathbf{\Lambda}^2$ converges to $\underline{\eta}-\overline{\eta}$ as well. In particular, subsequences $\left\|\mathbf{u}^{k_i}-\underline{\mathbf{u}}\right\|_\mathbf{\Lambda}^2-\left\|\mathbf{u}^{k_i}-\overline{\mathbf{u}}\right\|_\mathbf{\Lambda}^2$ and $\left\|\mathbf{u}^{k_i'}-\underline{\mathbf{u}}\right\|_\mathbf{\Lambda}^2-\left\|\mathbf{u}^{k_i'}-\overline{\mathbf{u}}\right\|_\mathbf{\Lambda}^2$ both converge to $\underline{\eta}-\overline{\eta}$, as $i$ goes to infinity. Noting that $\lim_{i\rightarrow\infty}\mathbf{u}^{k_i}=\underline{\mathbf{u}}$ and \eqref{8}, we obtain:
\begin{align}
&\underline{\eta}-\overline{\eta}\\
&=\lim_{i\rightarrow\infty}\left(\left\|\mathbf{u}^{k_i}-\underline{\mathbf{u}}\right\|_\mathbf{\Lambda}^2-\left\|\mathbf{u}^{k_i}-\overline{\mathbf{u}}\right\|_\mathbf{\Lambda}^2\right)\\
&=\lim_{i\rightarrow\infty}\bigg[c\mathbf{x}^{k_i\mathsf{T}}(\overline{\mathbf{x}}-\underline{\mathbf{x}})+(c+\rho)\mathbf{w}^{k_i\mathsf{T}}(\overline{\mathbf{w}}-\underline{\mathbf{w}})+\frac{1}{\rho}\boldsymbol{\alpha}^{k_i\mathsf{T}}(\overline{\boldsymbol{\alpha}}-\underline{\boldsymbol{\alpha}})\nonumber\\
&~~~+\frac{c}{2}\|\underline{\mathbf{x}}\|_2^2+\frac{c+\rho}{2}\|\underline{\mathbf{w}}\|_2^2+\frac{1}{2\rho}\|\underline{\boldsymbol{\alpha}}\|_2^2-\frac{c}{2}\|\overline{\mathbf{x}}\|_2^2-\frac{c+\rho}{2}\|\overline{\mathbf{w}}\|_2^2-\frac{1}{2\rho}\|\overline{\boldsymbol{\alpha}}\|_2^2\bigg]\\
&=c\underline{\mathbf{x}}^\mathsf{T}(\overline{\mathbf{x}}-\underline{\mathbf{x}})+(c+\rho)\underline{\mathbf{w}}^\mathsf{T}(\overline{\mathbf{w}}-\underline{\mathbf{w}})+\frac{1}{\rho}\underline{\boldsymbol{\alpha}}^\mathsf{T}(\overline{\boldsymbol{\alpha}}-\underline{\boldsymbol{\alpha}})\nonumber\\
&~~~+\frac{c}{2}\|\underline{\mathbf{x}}\|_2^2+\frac{c+\rho}{2}\|\underline{\mathbf{w}}\|_2^2+\frac{1}{2\rho}\|\underline{\boldsymbol{\alpha}}\|_2^2-\frac{c}{2}\|\overline{\mathbf{x}}\|_2^2-\frac{c+\rho}{2}\|\overline{\mathbf{w}}\|_2^2-\frac{1}{2\rho}\|\overline{\boldsymbol{\alpha}}\|_2^2.\label{9}
\end{align}
Similarly, from the perspective of $\mathbf{u}^{k_i'}$, we can get:
\begin{align}
&\underline{\eta}-\overline{\eta}\\
&=c\overline{\mathbf{x}}^\mathsf{T}(\overline{\mathbf{x}}-\underline{\mathbf{x}})+(c+\rho)\overline{\mathbf{w}}^\mathsf{T}(\overline{\mathbf{w}}-\underline{\mathbf{w}})+\frac{1}{\rho}\overline{\boldsymbol{\alpha}}^\mathsf{T}(\overline{\boldsymbol{\alpha}}-\underline{\boldsymbol{\alpha}})\nonumber\\
&~~~+\frac{c}{2}\|\underline{\mathbf{x}}\|_2^2+\frac{c+\rho}{2}\|\underline{\mathbf{w}}\|_2^2+\frac{1}{2\rho}\|\underline{\boldsymbol{\alpha}}\|_2^2-\frac{c}{2}\|\overline{\mathbf{x}}\|_2^2-\frac{c+\rho}{2}\|\overline{\mathbf{w}}\|_2^2-\frac{1}{2\rho}\|\overline{\boldsymbol{\alpha}}\|_2^2.\label{10}
\end{align}
Combining \eqref{9} and \eqref{10} yields:
\begin{align}
c\|\overline{\mathbf{x}}-\underline{\mathbf{x}}\|_2^2+(c+\rho)\|\overline{\mathbf{w}}-\underline{\mathbf{w}}\|_2^2+\frac{1}{\rho}\|\overline{\boldsymbol{\alpha}}-\underline{\boldsymbol{\alpha}}\|_2^2=0,
\end{align}
which implies $\underline{\mathbf{u}}=\overline{\mathbf{u}}$.

\section*{Appendix D\\Proof of Lemma \ref{lemD}}
For any $\mathbf{x,x'}\in\mathbb{R}^{np}$, we have:
\begin{align}
\left(\nabla f(\mathbf{x})-\nabla f(\mathbf{x}')\right)^\mathsf{T}(\mathbf{x}-\mathbf{x}')=\sum_{i=1}^n\left(\nabla f_i(\mathbf{x}_i)-\nabla f_i(\mathbf{x}_i')\right)^\mathsf{T}(\mathbf{x}_i-\mathbf{x}_i')\geq\sum_{i=1}^n\tau\|\mathbf{x}_i-\mathbf{x}_i'\|_2^2=\tau\left\|\mathbf{x}-\mathbf{x}'\right\|_2^2.
\end{align}
So, $f$ is strongly convex with constant $\tau$. Furthermore,
\begin{align}
&(\nabla g(\mathbf{y,z})-\nabla g(\mathbf{y',z'}))^\mathsf{T}\left(
\left[
\begin{array}{c}
\mathbf{y}\\
\mathbf{z}
\end{array}
\right]-
\left[
\begin{array}{c}
\mathbf{y}'\\
\mathbf{z}'
\end{array}
\right]
\right)\\
&=\sum_{i=1}^n\sum_{j\in\Omega_i}
\left[
\begin{array}{c}
\nabla_{\mathbf{y}_i}g_{ij}(\mathbf{y}_i,\mathbf{z}_{ij})-\nabla_{\mathbf{y}_i'}g_{ij}(\mathbf{y}_i',\mathbf{z}_{ij}')\\
\nabla_{\mathbf{z}_{ij}}g_{ij}(\mathbf{y}_i,\mathbf{z}_{ij})-\nabla_{\mathbf{z}_{ij}'}g_{ij}(\mathbf{y}_i',\mathbf{z}_{ij}')
\end{array}
\right]^\mathsf{T}
\left[
\begin{array}{c}
\mathbf{y}_i-\mathbf{y}_i'\\
\mathbf{z}_{ij}-\mathbf{z}_{ij}'
\end{array}
\right]\\
&\geq\sum_{i=1}^n\sum_{j\in\Omega_i}\tau\left\|
\left[
\begin{array}{c}
\mathbf{y}_i\\
\mathbf{z}_{ij}
\end{array}
\right]-
\left[
\begin{array}{c}
\mathbf{y}_i'\\
\mathbf{z}_{ij}'
\end{array}
\right]
\right\|_2^2\\
&=\tau\sum_{i=1}^n|\Omega_i|\|\mathbf{y}_i-\mathbf{y}_i'\|_2^2+\tau\sum_{i=1}^n\sum_{j\in\Omega_i}\|\mathbf{z}_{ij}-\mathbf{z}_{ij}'\|_2^2\\
&\geq\tau\left\|
\left[
\begin{array}{c}
\mathbf{y}\\
\mathbf{z}
\end{array}
\right]-
\left[
\begin{array}{c}
\mathbf{y}'\\
\mathbf{z}'
\end{array}
\right]
\right\|_2^2.
\end{align}
So, $g$ is strongly convex with constant $\tau$ as well.

\section*{Appendix E\\Proof of Theorem 2}
We first show the properness of the definitions of $\delta$ and $\beta$. Because $c>\rho+\frac{M^2}{2\tau}$, $\left(\frac{M^2}{c-\rho},2\tau\right)$ is a proper interval. On the interval $\left(\frac{M^2}{c-\rho},2\tau\right)$, the function of $\beta$ on the L.H.S. of \eqref{beta_def} decreases from some positive value to zero while the function of $\beta$ on the R.H.S. increases from zero to some positive value. This ensures that there exists a unique $\beta\in\left(\frac{M^2}{c-\rho},2\tau\right)$ such that Equation \eqref{beta_def} holds. As for $\delta$, since $c>\frac{L^2}{2\tau}$, the first term on the R.H.S. of \eqref{delta_def} is positive. Furthermore, because $\beta<2\tau$, the second term on the R.H.S. of \eqref{delta_def} is positive. Therefore, $\delta$ is positive.

Based on the definitions of $\delta$ and $\beta$, we have:
\begin{align}
&\left(\tau-\frac{L^2}{2c}-\frac{c\delta}{2}-\frac{3\rho\mu\Gamma^2\delta}{\mu-1}\right)\left\|\mathbf{x}^{k+1}-\mathbf{x}^*\right\|_2^2+\left(\tau-\frac{\beta}{2}-\frac{(c+\rho)\delta}{2}-\frac{3\rho\mu\delta}{\mu-1}-\frac{2M^2\mu\delta}{\rho}\right)\left\|\mathbf{w}^{k+1}-\mathbf{w}^*\right\|_2^2\nonumber\\
&~~~+\left(\frac{c-\rho}{2}-\frac{M^2}{2\beta}-\frac{3c^2\mu\delta}{\rho(\mu-1)}-\frac{2M^2\mu\delta}{\rho}\right)\left\|\mathbf{w}^k-\mathbf{w}^{k+1}\right\|_2^2+\frac{1-4\delta}{4\rho}\left\|\boldsymbol{\alpha}^{k+1}-\boldsymbol{\alpha}^k\right\|_2^2\geq0.
\end{align}
Hence,
\begin{align}
&\left(\tau-\frac{L^2}{2c}-\frac{c\delta}{2}-\frac{3\rho\mu\Gamma^2\delta}{\mu-1}\right)\left\|\mathbf{x}^{k+1}-\mathbf{x}^*\right\|_2^2+\left(\tau-\frac{\beta}{2}-\frac{(c+\rho)\delta}{2}-\frac{3\rho\mu\delta}{\mu-1}\right)\left\|\mathbf{w}^{k+1}-\mathbf{w}^*\right\|_2^2\nonumber\\
&~~~+\left(\frac{c-\rho}{2}-\frac{M^2}{2\beta}-\frac{3c^2\mu\delta}{\rho(\mu-1)}\right)\left\|\mathbf{w}^k-\mathbf{w}^{k+1}\right\|_2^2+\frac{1-4\delta}{4\rho}\left\|\boldsymbol{\alpha}^{k+1}-\boldsymbol{\alpha}^k\right\|_2^2\\
&\geq\frac{2M^2\mu\delta}{\rho}\left(\left\|\mathbf{w}^{k+1}-\mathbf{w}^*\right\|_2^2+\left\|\mathbf{w}^k-\mathbf{w}^{k+1}\right\|_2^2\right)\\
&\geq\frac{M^2\mu\delta}{\rho}\left\|\mathbf{w}^k-\mathbf{w}^*\right\|_2^2.
\end{align}
Equivalently,
\begin{align}
&\frac{c\delta}{2}\left\|\mathbf{x}^{k+1}-\mathbf{x}^*\right\|_2^2+\frac{(c+\rho)\delta}{2}\left\|\mathbf{w}^{k+1}-\mathbf{w}^*\right\|_2^2+\frac{\delta}{\rho}\left\|\boldsymbol{\alpha}^{k+1}-\boldsymbol{\alpha}^k\right\|_2^2+\frac{\delta\mu M^2}{\rho}\left\|\mathbf{w}^k-\mathbf{w}^*\right\|_2^2\nonumber\\
&~~~+\frac{3c^2\mu\delta}{\rho(\mu-1)}\left\|\mathbf{w}^{k+1}-\mathbf{w}^k\right\|_2^2+\frac{3\rho\mu\delta}{\mu-1}\left\|\mathbf{w}^{k+1}-\mathbf{w}^*\right\|_2^2+\frac{3\rho\mu\delta\Gamma^2}{\mu-1}\left\|\mathbf{x}^*-\mathbf{x}^{k+1}\right\|_2^2\nonumber\\
&\leq \tau\left\|\mathbf{x}^{k+1}-\mathbf{x}^*\right\|_2^2+\tau\left\|\mathbf{w}^{k+1}-\mathbf{w}^*\right\|_2^2+\frac{c-\rho}{2}\left\|\mathbf{w}^k-\mathbf{w}^{k+1}\right\|_2^2+\frac{1}{4\rho}\left\|\boldsymbol{\alpha}^{k+1}-\boldsymbol{\alpha}^k\right\|_2^2\nonumber\\
\label{11}&~~~-\frac{L^2}{2c}\left\|\mathbf{x}^{k+1}-\mathbf{x}^*\right\|_2^2-\frac{\beta}{2}\left\|\mathbf{w}^{k+1}-\mathbf{w}^*\right\|_2^2-\frac{M^2}{2\beta}\left\|\mathbf{w}^{k+1}-\mathbf{w}^k\right\|_2^2
\end{align}
For the last three terms of the L.H.S. of \eqref{11}, we note that:
\begin{align}
&\left\|c\left(\mathbf{w}^{k+1}-\mathbf{w}^k\right)+\rho\left(\mathbf{w}^{k+1}-\mathbf{w}^*\right)+\rho\mathbf{B}\left(\mathbf{x}^*-\mathbf{x}^{k+1}\right)\right\|_2^2\\
&\leq3c^2\left\|\mathbf{w}^{k+1}-\mathbf{w}^k\right\|_2^2+3\rho^2\left\|\mathbf{w}^{k+1}-\mathbf{w}^*\right\|_2^2+3\rho^2\left\|\mathbf{B}\left(\mathbf{x}^*-\mathbf{x}^{k+1}\right)\right\|_2^2\\
&\leq3c^2\left\|\mathbf{w}^{k+1}-\mathbf{w}^k\right\|_2^2+3\rho^2\left\|\mathbf{w}^{k+1}-\mathbf{w}^*\right\|_2^2+3\rho^2\Gamma^2\left\|\mathbf{x}^*-\mathbf{x}^{k+1}\right\|_2^2,
\end{align}
where the last step is due to the definition of $\Gamma$ (the spectral norm of $\mathbf{B}$). Substituting this result into \eqref{11} yields:
\begin{align}
&\frac{c\delta}{2}\left\|\mathbf{x}^{k+1}-\mathbf{x}^*\right\|_2^2+\frac{(c+\rho)\delta}{2}\left\|\mathbf{w}^{k+1}-\mathbf{w}^*\right\|_2^2+\frac{\delta}{\rho}\left\|\boldsymbol{\alpha}^{k+1}-\boldsymbol{\alpha}^k\right\|_2^2+\frac{\delta\mu M^2}{\rho}\left\|\mathbf{w}^k-\mathbf{w}^*\right\|_2^2\nonumber\\
&~~~+\frac{\mu\delta}{\rho(\mu-1)}\left\|c\left(\mathbf{w}^{k+1}-\mathbf{w}^k\right)+\rho\left(\mathbf{w}^{k+1}-\mathbf{w}^*\right)+\rho\mathbf{B}\left(\mathbf{x}^*-\mathbf{x}^{k+1}\right)\right\|_2^2\nonumber\\
&\leq \tau\left\|\mathbf{x}^{k+1}-\mathbf{x}^*\right\|_2^2+\tau\left\|\mathbf{w}^{k+1}-\mathbf{w}^*\right\|_2^2+\frac{c-\rho}{2}\left\|\mathbf{w}^k-\mathbf{w}^{k+1}\right\|_2^2+\frac{1}{4\rho}\left\|\boldsymbol{\alpha}^{k+1}-\boldsymbol{\alpha}^k\right\|_2^2\nonumber\\
\label{14}&~~~-\frac{L^2}{2c}\left\|\mathbf{x}^{k+1}-\mathbf{x}^*\right\|_2^2-\frac{\beta}{2}\left\|\mathbf{w}^{k+1}-\mathbf{w}^*\right\|_2^2-\frac{M^2}{2\beta}\left\|\mathbf{w}^{k+1}-\mathbf{w}^k\right\|_2^2.
\end{align}
Based on Assumption 2, Lemma \ref{lemA} and Equation \eqref{gg}, we have:
\begin{align}
&M^2\left\|\mathbf{w}^k-\mathbf{w}^*\right\|_2^2\\
&\geq\left\|\nabla g\left(\mathbf{w}^k\right)-\nabla g\left(\mathbf{w}^*\right)\right\|_2^2\\
&=\left\|c\left(\mathbf{w}^{k+1}-\mathbf{w}^k\right)+\boldsymbol{\alpha}^*-\boldsymbol{\alpha}^k+\rho\left(\mathbf{w}^{k+1}-\mathbf{w}^*\right)+\rho\mathbf{B}\left(\mathbf{x}^*-\mathbf{x}^{k+1}\right)\right\|_2^2\\
&\geq\frac{1}{\mu}\left\|\boldsymbol{\alpha}^*-\boldsymbol{\alpha}^k\right\|_2^2-\frac{1}{\mu-1}\left\|c\left(\mathbf{w}^{k+1}-\mathbf{w}^k\right)+\rho\left(\mathbf{w}^{k+1}-\mathbf{w}^*\right)+\rho\mathbf{B}\left(\mathbf{x}^*-\mathbf{x}^{k+1}\right)\right\|_2^2,
\end{align}
where $\mu>1$ is any constant greater than 1 and the last step makes use of the fact: $\|\mathbf{x+y}\|_2^2\geq\frac{1}{\mu}\|\mathbf{y}\|_2^2-\frac{1}{\mu-1}\|\mathbf{x}\|_2^2$ for any $\mu>1$ and any vectors $\mathbf{x,y}$. Thereby,

\begin{align}
\left\|\boldsymbol{\alpha}^*-\boldsymbol{\alpha}^k\right\|_2^2\leq\mu M^2\left\|\mathbf{w}^k-\mathbf{w}^*\right\|_2^2+\frac{\mu}{\mu-1}\left\|c\left(\mathbf{w}^{k+1}-\mathbf{w}^k\right)+\rho\left(\mathbf{w}^{k+1}-\mathbf{w}^*\right)+\rho\mathbf{B}\left(\mathbf{x}^*-\mathbf{x}^{k+1}\right)\right\|_2^2.\label{12}
\end{align}

Using \eqref{12}, we obtain:
\begin{align}
&\left\|\boldsymbol{\alpha}^{k+1}-\boldsymbol{\alpha}^*\right\|_2^2\\
&\leq 2\left\|\boldsymbol{\alpha}^{k+1}-\boldsymbol{\alpha}^k\right\|_2^2+2\left\|\boldsymbol{\alpha}^k-\boldsymbol{\alpha}^*\right\|_2^2\\
&\leq 2\left\|\boldsymbol{\alpha}^{k+1}-\boldsymbol{\alpha}^k\right\|_2^2+2\mu M^2\left\|\mathbf{w}^k-\mathbf{w}^*\right\|_2^2\nonumber\\
\label{13}&~~~+\frac{2\mu}{\mu-1}\left\|c\left(\mathbf{w}^{k+1}-\mathbf{w}^k\right)+\rho\left(\mathbf{w}^{k+1}-\mathbf{w}^*\right)+\rho\mathbf{B}\left(\mathbf{x}^*-\mathbf{x}^{k+1}\right)\right\|_2^2.
\end{align}
By exploiting Cauchy's inequality and Assumption 2, Lemma \ref{lemA}, we get:
\begin{align}
&-\left(\mathbf{x}^{k+1}-\mathbf{x}^*\right)^\mathsf{T}\left(\nabla f\left(\mathbf{x}^{k+1}\right)-\nabla f\left(\mathbf{x}^k\right)\right)\\
&\geq\frac{L^2}{2c}\left\|\mathbf{x}^{k+1}-\mathbf{x}^*\right\|_2^2-\frac{c}{2L^2}\left\|\nabla f\left(\mathbf{x}^{k+1}\right)-\nabla f\left(\mathbf{x}^k\right)\right\|_2^2\\
&\label{e1}\geq\frac{L^2}{2c}\left\|\mathbf{x}^{k+1}-\mathbf{x}^*\right\|_2^2-\frac{c}{2}\left\|\mathbf{x}^{k+1}-\mathbf{x}^k\right\|_2^2,
\end{align}
and
\begin{align}
&-\left(\mathbf{w}^{k+1}-\mathbf{w}^*\right)^\mathsf{T}\left(\nabla g\left(\mathbf{w}^{k+1}\right)-\nabla g\left(\mathbf{w}^k\right)\right)\\
&\geq-\frac{\beta}{2}\left\|\mathbf{w}^{k+1}-\mathbf{w}^*\right\|_2^2-\frac{1}{2\beta}\left\|\nabla g\left(\mathbf{w}^{k+1}\right)-\nabla g\left(\mathbf{w}^k\right)\right\|_2^2\\
&\label{e2}\geq-\frac{\beta}{2}\left\|\mathbf{w}^{k+1}-\mathbf{w}^*\right\|_2^2-\frac{M^2}{2\beta}\left\|\mathbf{w}^{k+1}-\mathbf{w}^k\right\|_2^2.
\end{align}
Therefore,
\begin{align}
&\delta\left\|\mathbf{u}^{k+1}-\mathbf{u}^*\right\|_\mathbf{\Lambda}^2\\
&=\frac{\delta c}{2}\left\|\mathbf{x}^{k+1}-\mathbf{x}^*\right\|_2^2+\frac{\delta(\rho+c)}{2}\left\|\mathbf{w}^{k+1}-\mathbf{w}^*\right\|_2^2+\frac{\delta}{2\rho}\left\|\boldsymbol{\alpha}^{k+1}-\boldsymbol{\alpha}^*\right\|_2^2\\
&\leq \frac{\delta c}{2}\left\|\mathbf{x}^{k+1}-\mathbf{x}^*\right\|_2^2+\frac{\delta(\rho+c)}{2}\left\|\mathbf{w}^{k+1}-\mathbf{w}^*\right\|_2^2+\frac{\delta}{\rho}\left\|\boldsymbol{\alpha}^{k+1}-\boldsymbol{\alpha}^k\right\|_2^2+\frac{\delta\mu M^2}{\rho}\left\|\mathbf{w}^k-\mathbf{w}^*\right\|_2^2\nonumber\\
\label{13'}&~~~+\frac{\delta\mu}{\rho(\mu-1)}\left\|c\left(\mathbf{w}^{k+1}-\mathbf{w}^k\right)+\rho\left(\mathbf{w}^{k+1}-\mathbf{w}^*\right)+\rho\mathbf{B}\left(\mathbf{x}^*-\mathbf{x}^{k+1}\right)\right\|_2^2\\
&\leq\tau\left\|\mathbf{x}^{k+1}-\mathbf{x}^*\right\|_2^2+\tau\left\|\mathbf{w}^{k+1}-\mathbf{w}^*\right\|_2^2+\frac{c-\rho}{2}\left\|\mathbf{w}^k-\mathbf{w}^{k+1}\right\|_2^2+\frac{1}{4\rho}\left\|\boldsymbol{\alpha}^{k+1}-\boldsymbol{\alpha}^k\right\|_2^2\nonumber\\
\label{14'}&~~~-\frac{L^2}{2c}\left\|\mathbf{x}^{k+1}-\mathbf{x}^*\right\|_2^2-\frac{\beta}{2}\left\|\mathbf{w}^{k+1}-\mathbf{w}^*\right\|_2^2-\frac{M^2}{2\beta}\left\|\mathbf{w}^{k+1}-\mathbf{w}^k\right\|_2^2.\\
&\leq\tau\left\|\mathbf{x}^{k+1}-\mathbf{x}^*\right\|_2^2+\tau\left\|\mathbf{w}^{k+1}-\mathbf{w}^*\right\|_2^2+\frac{c}{2}\left\|\mathbf{x}^{k+1}-\mathbf{x}^k\right\|_2^2+\frac{c-\rho}{2}\left\|\mathbf{w}^k-\mathbf{w}^{k+1}\right\|_2^2\nonumber\\
&~~~+\frac{1}{4\rho}\left\|\boldsymbol{\alpha}^{k+1}-\boldsymbol{\alpha}^k\right\|_2^2-\left(\mathbf{x}^{k+1}-\mathbf{x}^*\right)^\mathsf{T}\left(\nabla f\left(\mathbf{x}^{k+1}\right)-\nabla f\left(\mathbf{x}^k\right)\right)\nonumber\\
&~~~-\left(\mathbf{w}^{k+1}-\mathbf{w}^*\right)^\mathsf{T}\left(\nabla g\left(\mathbf{w}^{k+1}\right)-\nabla g\left(\mathbf{w}^k\right)\right),\label{e3}
\end{align}
where \eqref{13'} is due to \eqref{13}; \eqref{14'} comes from \eqref{14}; and \eqref{e3} is because of \eqref{e1} and \eqref{e2}. Due to Assumption 3 and Lemma \ref{lemD}, we have:
\begin{align}
&\tau\left\|\mathbf{x}^{k+1}-\mathbf{x}^*\right\|_2^2\\
&\leq\left(\nabla f\left(\mathbf{x}^{k+1}\right)-\nabla f\left(\mathbf{x}^*\right)\right)^\mathsf{T}\left(\mathbf{x}^{k+1}-\mathbf{x}^*\right)\\
&=\left(\mathbf{x}^{k+1}-\mathbf{x}^*\right)^\mathsf{T}\left(\nabla f\left(\mathbf{x}^k\right)-\nabla f\left(\mathbf{x}^*\right)\right)+\left(\mathbf{x}^{k+1}-\mathbf{x}^*\right)^\mathsf{T}\left(\nabla f\left(\mathbf{x}^{k+1}\right)-\nabla f\left(\mathbf{x}^k\right)\right),\label{ee1}
\end{align}
and,
\begin{align}
&\tau\left\|\mathbf{w}^{k+1}-\mathbf{w}^*\right\|_2^2\\
&\leq\left(\nabla g\left(\mathbf{w}^{k+1}\right)-\nabla g\left(\mathbf{w}^*\right)\right)^\mathsf{T}\left(\mathbf{w}^{k+1}-\mathbf{w}^*\right)\\
&=\left(\mathbf{w}^{k+1}-\mathbf{w}^*\right)^\mathsf{T}\left(\nabla g\left(\mathbf{w}^k\right)-\nabla g\left(\mathbf{w}^*\right)\right)+\left(\mathbf{w}^{k+1}-\mathbf{w}^*\right)^\mathsf{T}\left(\nabla g\left(\mathbf{w}^{k+1}\right)-\nabla g\left(\mathbf{w}^k\right)\right),\label{ee2}
\end{align}

Adding \eqref{ee1} and \eqref{ee2} and using \eqref{3}, we get:
\begin{align}
&\tau\left\|\mathbf{x}^{k+1}-\mathbf{x}^*\right\|_2^2+\tau\left\|\mathbf{w}^{k+1}-\mathbf{w}^*\right\|_2^2\nonumber\\
&\leq \left\|\mathbf{u}^k-\mathbf{u}^*\right\|_\mathbf{\Lambda}^2-\left\|\mathbf{u}^{k+1}-\mathbf{u}^*\right\|_\mathbf{\Lambda}^2-\frac{c}{2}\left\|\mathbf{x}^k-\mathbf{x}^{k+1}\right\|_2^2-\frac{c-\rho}{2}\left\|\mathbf{w}^k-\mathbf{w}^{k+1}\right\|_2^2-\frac{1}{4\rho}\left\|\boldsymbol{\alpha}^{k+1}-\boldsymbol{\alpha}^k\right\|_2^2\nonumber\\
&~~~+\left(\mathbf{x}^{k+1}-\mathbf{x}^*\right)^\mathsf{T}\left(\nabla f\left(\mathbf{x}^{k+1}\right)-\nabla f\left(\mathbf{x}^k\right)\right)+\left(\mathbf{w}^{k+1}-\mathbf{w}^*\right)^\mathsf{T}\left(\nabla g\left(\mathbf{w}^{k+1}\right)-\nabla g\left(\mathbf{w}^k\right)\right).\label{k1}
\end{align}

Combining \eqref{e3} and \eqref{k1} leads to:
\begin{align}
\delta\left\|\mathbf{u}^{k+1}-\mathbf{u}^*\right\|_\mathbf{\Lambda}^2\leq\left\|\mathbf{u}^k-\mathbf{u}^*\right\|_\mathbf{\Lambda}^2-\left\|\mathbf{u}^{k+1}-\mathbf{u}^*\right\|_\mathbf{\Lambda}^2,
\end{align}
which is tantamount to \eqref{linear_convergence}.

\bibliography{mybib}{}
\bibliographystyle{ieeetr}

\end{document}